\documentclass[11pt]{amsart}[2000]

\usepackage{amsmath}
\usepackage{amssymb}
\usepackage{amsthm} 
\usepackage{amscd} 
\usepackage[all]{xy}
\usepackage{enumitem}
\usepackage{mathtools}
\usepackage{color}
\usepackage{bbm}
\usepackage{tikz-cd}
\usepackage{soul}
\usepackage{xfrac}

\theoremstyle{plain}

\newtheorem{Thm}{Theorem}[section]
\newtheorem{Cor}[Thm]{Corollary}
\newtheorem{Lemma}[Thm]{Lemma}

\newtheorem{Prop}[Thm]{Proposition}

\theoremstyle{definition}
\newtheorem{Def}[Thm]{Definition}
\newtheorem{Not}[Thm]{Notation}

\newtheorem{Exl}[Thm]{Example}
\newtheorem{Rmk}[Thm]{Remark}

\numberwithin{equation}{section}

\newcommand{\A}{{\mathcal{A}}}
\newcommand{\B}{{\mathcal{B}}}
\newcommand{\C}{{\mathcal{C}}}

\newcommand{\G}{{\mathcal{G}}}

\newcommand{\J}{{\mathcal{J}}}
\newcommand{\K}{{\mathcal{K}}}
\renewcommand{\L}{{\mathcal{L}}}
\newcommand{\M}{{\mathcal{M}}}

\newcommand{\T}{{\mathcal{T}}}

\newcommand{\bN}{\mathbb{N}}
\newcommand{\bT}{\mathbb{T}}
\newcommand{\bZ}{\mathbb{Z}}

\newcommand{\bbK}{{\mathbb{K}}}
\newcommand{\bbN}{{\mathbb{N}}}

\newcommand{\bbT}{{\mathbb{T}}}
\newcommand{\bbZ}{{\mathbb{Z}}}

\newcommand{\Aut}{\operatorname{Aut}}
\newcommand{\alg}{\operatorname{alg}}

\newcommand{\id}{{\operatorname{id}}}

\newcommand{\Span}{\operatorname{span}}

\newcommand{\ca}{\mathrm{C}^*}

\newcommand{\ga}{{\gamma^A}}

\newcommand{\hga}{{\widehat{\gamma}^A}}

\makeatletter
\@namedef{subjclassname@1991}{2020 Mathematics Subject Classification}
\makeatother

\begin{document}
	
	\title[Shift equivalences and Cuntz-Krieger algebras]{Shift equivalences through the lens of Cuntz-Krieger algebras}
	
	\author[T. M. Carlsen]{Toke Meier Carlsen}
	\address{Department of Mathematics\\ University of Faroe Islands \\ T\'orshavn \\ Faroe Islands.}
	\email{toke.carlsen@gmail.com}
	
	\author[A. Dor-On]{Adam Dor-On}
	\address{Mathematisches Institut \\ WWU M\"{u}nster \\ M\"{u}nster \\ Germany.}
	\email{adoron.math@gmail.com}
	
	\author[S. Eilers]{S{\o{}}ren Eilers}
	\address{Department of Mathematical Sciences\\ University of Copenhagen \\ Copenhagen \\ Denmark.}
	\email{eilers@math.ku.dk}

	
	\thanks{The first author was supported by the Research Council of the Faroe Islands. The second author was supported by NSF grant DMS-1900916 and the European Union's Horizon 2020 Marie Sk\l{}odowska-Curie grant 839412 - IRIOA. The third author was supported by the DFF-Research Project 2 `Automorphisms and Invariants of Operator Algebras', no. 7014-00145B. Finally, this project was partially funded by the Deutsche Forschungsgemeinschaft (DFG, German Research Foundation) under Germany’s Excellence Strategy – EXC 2044 – 390685587, Mathematics Münster – Dynamics – Geometry – Structure; the Deutsche Forschungsgemeinschaft (DFG, German Research Foundation) – Project-ID 427320536 – SFB 1442, and ERC Advanced Grant 834267 - AMAREC}

	\begin{abstract}
	
	Motivated by Williams' problem of measuring novel differences between shift equivalence (SE) and strong shift equivalence (SSE), we introduce three equivalence relations that provide new ways to obstruct SSE while merely assuming SE.
	
Our shift equivalence relations arise from studying graph C*-algebras, where a variety of intermediary equivalence relations naturally arise. As a consequence we realize a goal sought after by Muhly, Pask and Tomforde, measure a delicate difference between SSE and SE in terms of Pimsner dilations for C*-correspondences of adjacency matrices, and use this distinction to refute a proof from a previous paper.

	\end{abstract}
	
		
	\maketitle

	\renewcommand{\O}{{\mathcal{O}}}
	
\section{Introduction}

Initially recognized in the 40's as the right object to model quantum phenomena, C*-algebras are applied today in a variety of areas including theoretical physics, topology, differential geometry and dynamical systems. Such applications drive the impetus for obtaining structural and classification results for C*-algebras, especially in relation with Elliott's classification programme \cite{Ell93, WZ10, Win14, TWW17}. One fantastic application of C*-algebras in dynamics, using tools from classification of C*-algebras, is the classification of Cantor minimal systems up to orbit equivalence by their dimension groups \cite{GPS95, GMPS10}. Similar to this, our work here deals with subtle invariants arising from C*-algebras associated to subshifts of finite type (SFTs), with the aim of distinguishing SFTs up to conjugacy.
	
In a seminal 1973 paper \cite{Wil73}, Williams recast conjugacy and eventual conjugacy for SFTs purely in terms of equivalence relations between adjacency matrices of the directed graphs. These are called strong shift equivalence (SSE) and shift equivalence (SE) respectively. Williams expected SSE and SE to be the same \cite[Proposition 7.2]{Wil73}, but after around 25 years the last hope for a positive answer to Williams' problem, even under the most restrictive conditions, was extinguished by Kim and Roush \cite{KR99}. Although these counterexamples are concrete, aperiodic and irreducible $7 \times 7$ matrices, showing that they are not strong shift equivalent requires an invariant which is very difficult to compute. Thus, finding new obstructions to strong shift equivalence when two matrices are only assumed to be shift equivalent is an important endeavor, even just for $2 \times 2$ matrices (see \cite[Example 7.3.13]{LM95}).

\begin{Def}[\cite{Wil73}]
Let $A$ and $B$ be matrices indexed by sets $V$ and $W$ respectively, with (possibly infinite) cardinal entries. We say that $A$ and $B$ are 

\begin{enumerate}
\item \emph{shift equivalent} with lag $m \in \mathbb{N} \setminus \{0\}$ if there are a $V \times W$ matrix $R$ and a $W\times V$ matrix $S$ with cardinal entries such that
\begin{align*}
A^m = RS, \ \ B^m = SR, \\
SA = BS, \ \ AR = RB.
\end{align*}

\item \emph{elementary shift related} if they are shift equivalent with lag $1$.

\item \emph{strong shift equivalent} if they are equivalent in the transitive closure of elementary shift relation. 
\end{enumerate}
\end{Def}

In tandem with early attacks on Williams problem, Cuntz and Krieger \cite{CK80} created a bridgehead between symbolic dynamics and operator algebras, where several natural properties of subshifts of finite type are expressed through associated C*-algebras. In fact, by \cite[Proposition 2.17]{CK80} we know that strong shift equivalence of $A$ and $B$ implies that the Cuntz-Krieger C*-algebras $\mathcal{O}_A$ and $\mathcal{O}_B$ are stably isomorphic in a way preserving their gauge actions $\gamma^A$ and $\gamma^B$ and their diagonal subalgebras $\mathcal{D}_A$ and $\mathcal{D}_B$. On the other hand, by a theorem of Krieger \cite{Kri80} we know that the dimension group triples of SFTs are isomorphic if and only if the associated matrices are SE. Since these dimension group triples coincide with $K$-theoretical data of crossed products of Cuntz-Krieger C*-algebras by their gauge action, Krieger's theorem implies as a corollary (see Section \ref{s:final}) that if two Cuntz-Krieger algebras $\mathcal{O}_A$ and $\mathcal{O}_B$ are stably isomorphic in a way preserving their gauge actions $\gamma^A$ and $\gamma^B$, then their defining adjacency matrices $A$ and $B$ are shift equivalent. Through the lens of Cuntz-Krieger algebras, this provides several natural equivalence relations between SSE and SE, and it then becomes important to orient them and determine whether they coincide with SSE, SE or perhaps a completely new equivalence relation strictly between SSE and SE. Such distinctions may pave the way towards more concrete and computable invariants that distinguish SFTs up to conjugacy.

In this paper we introduce, study and orient three equivalence relations that provide new ways of measuring the difference between SSE and SE. Before we discuss these, let us first mention the state of the art.

A partial converse of the corollary to Krieger's theorem was obtained by Bratteli and Kishimoto \cite{BK00}, using deep machinery from C*-algebra K-theory classification, involving the essential concept of Rokhlin towers. More precisely, using their work it can be shown that if $A$ and $B$ are two \emph{aperiodic and irreducible} adjacency matrices then
\begin{gather} \label{eq:krieg-conv}
A \ \text{is SE to} \ B \Longleftrightarrow ({\mathcal O}_A\otimes\mathbb K,\gamma^A\otimes \operatorname{id})\simeq  ({\mathcal O}_B\otimes\mathbb K,\gamma^B\otimes \operatorname{id}).
\end{gather}
This \emph{converse to Krieger's corollary} for essential matrices remains a subtle and important classification problem in this line of research, and is one of the key motivating problems for our work here.

After a major undertaking pioneered by Matsumoto \cite{km:oetmscka,km:ucoemsgacka}, it is now known that several key concepts in symbolic dynamics may be fully understood in terms of operator algebraic descriptions. Indeed, both SSE as well as flow equivalence (FE) of subshifts of finite type  may be given an inherently operator algebraic characterization. Suppose $A$ and $B$ are finite adjacency matrices defining two-sided SFTs $(X_A,\sigma_A)$ and $(X_B,\sigma_B)$ respectively (see Section 2). When $A$ and $B$ are irreducible, Matsumoto and Matui \cite{MM14} established that
\begin{gather*}
(X_A,\sigma_A) \text{ is FE to } (X_B,\sigma_B) \ \Longleftrightarrow\\ ({\mathcal O}_A\otimes\mathbb K,\mathcal D_A\otimes c_0)\simeq  ({\mathcal O}_B\otimes\mathbb K,\mathcal D_B\otimes c_0)
\end{gather*}
where $\bbK$ is the C*-algebra of compact operators on $\ell^2(\mathbb{N})$ and $c_0$ is its subalgebra of diagonal operators. This was later extended to cover all two-sided SFTs by the first and third named authors with Ortega and Restorff \cite{CEOR19}.

For conjugacy, the first named author and Rout \cite{CR17} proved that $(X_A,\sigma_A)$  is conjugate to $(X_B,\sigma_B)$ precisely when the associated Cuntz-Krieger algebras ${\mathcal O}_A$ and ${\mathcal O}_A$ are stably isomorphic in a way preserving both the gauge actions and the diagonals (this time with no additional restrictions on $A$ and $B$). More precisely, and combining with Williams' characterization of conjugacy, we have 
\begin{gather} \label{eq:williams-conv}
A \ \text{is SSE to} \ B \Longleftrightarrow \\
 ({\mathcal O}_A\otimes\mathbb K,\gamma^A\otimes \operatorname{id},\mathcal D_A\otimes c_0)\simeq  ({\mathcal O}_B\otimes\mathbb K,\gamma^B\otimes \operatorname{id},\mathcal D_B\otimes c_0).\notag
 \end{gather}

The equivalences in equations \eqref{eq:krieg-conv} and \eqref{eq:williams-conv} provide a fresh perspective for symbolic dynamics via C*-algebras, and reinterprets Williams' problem via the counterexamples provided by Kim and Roush \cite{KR99,KRred}. More precisely, the examples of Kim and Roush show that two adjacency matrices can be shift equivalent and with flow equivalent two-sided SFTs without being strong shift equivalent. Hence, in terms of C*-algebras this shows it is not always possible to trade in two isomorphisms -- one respecting the diagonal, and one respecting the gauge action -- for one which respects both.

Another research agenda motivating this work is the \emph{graded isomorphism problem} of Hazrat from the theory of Leavitt path algebras, where one studies graded isomorphisms of pure algebras by means of graded K-theory. In work of Hazrat \cite{rh:ggclpa}, it was shown that the relevant graded K-theoretical data is in fact the same as Krieger's dimension group triple, and Hazrat conjectured that this invariant is complete when the class of the unit is added to the dimension triple as part of the invariant. Thus, the converse to Krieger's corollary is a topological analogue of Hazrat's conjecture. Although substantial advances have been made \cite{ALH+,AP14}, Hazrat's conjecture remains elusive.

For a $V \times W$ matrix $F=[F_{ij}]$ with cardinal entries, we denote
$$
E_F:= \{ \ (v,w,\alpha) \ | \ 0 \leq \alpha < F_{vw}, \ v\in V, \ w\in W \ \},
$$
so that $r(v,w,\alpha) = w$ and $s(v,w,\alpha) = v$, $\alpha$ is an ordinal, and $F_{vw}$ is interpreted as the least ordinal with cardinality $F_{vw}$. When $V=W$, this makes $G_F := (V,E_F)$ into a directed graph in its own right. For two matrices $C$ over $V \times W$ and $D$ over $W\times X$ with cardinal entries, we denote the \emph{fibered product}
$$
E_C \times E_D := \{ \ cd \ | \ c \in E_C, \ d\in E_D, \ r(c) = s(d) \ \}.
$$
Note here that we write $cd$ to mean the pair $(c,d)$ for which $r(c) = s(d)$, which should be thought of as concatenation of edges, even if there are no actual graphs. When $V=W$, for any $n\in \bN$ we denote $E_C^n$ the $n$-fold product of $E_C$ with itself, so that $E_C^n$ is naturally identified with $E_{C^n}$.

One might also consider $E_C \times E_D$ as a pullback of $W$, and indicate the role of $W$ in the notation. We refrain from doing so for purposed of readability.

The following definition was dubbed ``specified equivalence" by Nasu \cite{Nas95} in his study of shift equivalences between textile systems. 

\begin{Def} \label{d:path-isom}
Let $A$ and $B$ be matrices with cardinal entries over $V\times W$. A \emph{path isomorphism} is a bijection $\phi : E_A \rightarrow E_B$ such that $s(\phi(e)) = s(e)$ and $r(\phi(e)) = r(e)$ for every $e\in E_A$.
\end{Def}

Using path isomorphism, we define our first equivalence relation of \emph{compatible shift equivalence} (CSE) (see Definition \ref{d:cse}). The mere existence of path isomorphisms $\phi_R, \phi_S, \psi_A, \psi_B$ in Definition \ref{d:cse} implies shift equivalence, so that CSE is a direct strengthening of shift equivalence. In Theorem \ref{t:main} we show that CSE coincides with SSE, and this allows for a direct comparison between SSE and SE via CSE. Motivated by Williams' problem, another related notion called \emph{adapted shift equivalence} was studied by Parry \cite{Par81}, and was also shown to be equivalent to SSE. More precisely, instead of a requirement on compatibility of path isomorphisms between shift equivalent matrices, adpated shift equivalence is a shift equivalence of lag $m \in \mathbb{N} \setminus \{0\}$ between the adjacency matrices of the $m$-line graphs of $A$ and $B$.

It is known that shift equivalence is decidable by Kim and Roush \cite{KRdeci, KRdecii}, but the problem of decidability of strong shift equivalence remains a fundamental open problem in symbolic dynamics. In fact, this was the original motivation in Williams' paper \cite{Wil73}. The identification between path isomorphisms in the definition of CSE illustrates in what way the algorithm of Kim and Roush would need to improve were SSE turn out to be decidable. 

Our second equivalence relation is called \emph{representable shift equivalence} (RSE) (see Definition \ref{d:rse}). It arises naturally when one attempts to represent shift equivalence as bounded operators on Hilbert space. 

Surprisingly, by merely representing the relations of shift equivalence as bounded operator on Hilbert space, we get SSE (see Theorem \ref{t:main}). Considering the counterexamples of Kim and Roush, we see that if $A$ and $B$ are shift equivalent but not strong shift equivalent, then it follows that one of the four relations of shift equivalence must fail when representing everything on the same Hilbert space (see Section \ref{s:rse} for more details).

In what follows, we will say that a matrix is \emph{essential} if it has no zero rows and no zero columns. In the work of Pimsner \cite{Pim95}, Pimsner dilations were introduced and were subsequently studied by several authors \cite{MPT08, KK14, EKK17, Ery+}. Pimsner dilations offer a ``reversible" perspective for Cuntz-Pimsner C*-algebras, showing that they are always generated by an imprimitivity bimodule that contains the original C*-correspondence. In \cite[Remark 5.5]{MPT08} an equivalent formulation of strong shift equivalence in terms of Pimsner dilations was sought. This leads us to our third equivalence relation, which we call \emph{strong Morita shift equivalence} (SMSE) (see Definition \ref{d:smse}). Let $A$ and $B$ be finite essential matrices. Denote by $X(A)_{\infty}$ and $X(B)_{\infty}$ the Pimsner dilations of their graph C*-correspondences (see Section \ref{s:SE-CP}), and by $C^*(G_A)$ and $C^*(G_B)$ their Cuntz-Kriger graph C*-algebras. Our main theorem provides an equivalent formulation sought by Muhly, Pask and Tomforde, and orients CSE and RSE in one fell swoop (see Theorem \ref{t:main}). 

\begin{Thm}
Suppose $A$ and $B$ are two finite essential matrices with entries in $\mathbb{N}$. Then the following are equivalent,
\begin{enumerate}
\item $A$ and $B$ are strong shift equivalent.
\item $A$ and $B$ are compatible shift equivalent.
\item $A$ and $B$ are representable shift equivalent.
\item $X(A)_{\infty}$ and $X(B)_{\infty}$ are strong Morita shift equivalent.
\item $C^*(G_A)$ and $C^*(G_B)$ are equivariantly stably isomorphic in a way that respects the diagonals.
\end{enumerate}
\end{Thm}

For some time now, several experts have been perplexed about certain Pimsner dilation techniques (see for instance \cite[Remark 5.5]{MPT08}, the incorrect proof of \cite[Theorem 5.8]{KK14}, the subsequent corrigendum \cite{KK14}, 
and the recent revision of \cite{Ery+}), and many are still wondering whether they can be used to show that shift equivalence of $A$ and $B$ implies strong Morita equivalence of the Pimsner dilations $X(A)_{\infty}$ and $X(B)_{\infty}$ (in the sense of Abadie, Eilers and Exel \cite[Section 4]{AEE98}, or Muhly and Solel \cite{MS00}). The importance of this question is further elevated because of Theorem \ref{t:down-and-up} and the discussion succeeding it, where we show that a positive answer to it is \emph{equivalent} to the \emph{converse to Krieger's corollary}. Combining our main theorem with the celebrated counterexamples of Kim and Roush \cite{KR99}, as well as the work of Bratteli and Kishimoto \cite{BK00}, we obtain the following cutoff result. This result addresses the ambiguities mentioned above, and refutes the proof of \cite[Theorem 5.8]{KK14} (see Theorem \ref{t:cutoff}).

\begin{Thm}
There exist finite aperiodic $7 \times 7$ irreducible matrices $A$ and $B$ with entries in $\mathbb{N}$ such that $X(A)_{\infty}$ and $X(B)_{\infty}$ are strong Morita equivalent, but not strong Morita shift equivalent.
\end{Thm}

Since the validity of \cite[Theorem 5.8]{KK14} is in question, so is the validity of \cite[Corollary 5.11]{KK14}. This latter result states that shift equivalence of $A$ and $B$ implies the (not necessarily equivariant) stable isomorphisms of $C^*(G_A)$ and $C^*(G_B)$. However, thanks to Ara, Hazrat and Li \cite{ALH+} this result can still be recovered. Indeed, in work of the third author with Restorff, Ruiz and S\o rensen \cite{ERRS+} it was shown that filtered, ordered K-theory classifies unital graph C*-algebras up to stable isomorphism. By showing that filtered ordered K-theory is an invariant of shift equivalence, Ara, Hazrat and Li \cite{ALH+} ipso facto show that shift equivalence implies the stable isomorphisms of unital graph C*-algebras.

This paper contains 8 sections, including this introductory section. In the preliminaries Section \ref{s:prelim} we discuss some of the basic theory of directed graphs, subshifts of finite type, groupoid C*-algebra description of graph C*-algebras and Cuntz-Krieger C*-algebras. In Section \ref{s:SE-CP} we discuss some of the theory of C*-correspondences, Cuntz-Pimsner algebras, Pimsner dilations and equivariant isomorphisms. We provide there a characterization of the existence of an equivariant stable isomorphism between Cuntz-Pimsner algebras in terms of Pimsner dilations. In Section \ref{s:cse} we introduce CSE, show it is an equivalence relation, and that it is implied by SSE. In Section \ref{s:rse} we introduce RSE, show that it is implied by CSE, and upgrade representations to be faithful on associated graph C*-algebras. In Section \ref{s:smse} we explain how to concretely construct inductive limits related to Pimsner dilations, introduce SMSE, show that RSE implies SMSE and show that SMSE implies the existence of a stable equivariant diagonal-preserving isomorphism between graph C*-algebras. Finally in Section \ref{s:final} we discuss dimension triples for graph C*-algebras, orient different equivalence relations on adjacency matrices and prove the main results stated above. 

\section{Preliminaries} \label{s:prelim}

In this section we explain some of the basic theory of directed graphs, subshifts of finite type, groupoid C*-algebra descriptions of graph C*-algebras, Cuntz-Krieger algebras and dimension triples. We recommend \cite{LM95} for the basics of symbolic dynamics.

A directed graph $G=(V,E,r,s)$ is comprised of a vertex set $V$ and an edge set $E$ together with range and source maps $r,s : E \rightarrow V$. 

We say that a directed graph $G=(V,E)$ has \emph{finite out-degrees} if $s^{-1}(v)$ is finite for every $v\in V$. We say that $G$ is \emph{finite} if both $V$ and $E$ are finite. When $G$ has finite out-degrees and has no sources (i.e., $r$ is surjective) and no sinks (i.e., $s$ is surjective), we may define the \emph{two-sided edge shift} to be the pair $(X_E,\sigma_E)$ where $X_E$ is the set of bi-infinite paths
$$
X_E = \{ \ (e_n)_{n \in \mathbb{Z}} \ | \ s(e_{i+1}) = r(e_i) \ \} \subseteq \prod_{i\in \mathbb{Z}}E
$$
with the product topology, and $\sigma_E : X_E \rightarrow X_E$ is the left shift homeomorphism given by $\sigma_E((e_n)_{n \in \mathbb{Z}}) = (e_{n+1})_{n\in \mathbb{Z}}$. The \emph{one-sided edge shift} $(X_E^+, \sigma_E^+)$ is defined the same as above, by replacing every occurrence of $\mathbb{Z}$ with $\mathbb{N}$.

Given a directed graph $G=(V,E)$, we may always form its $V\times V$ adjacency matrix with cardinal entries given for $v,w\in V$ by
$$
A_E(v,w) = | \{ \ e\in E \ | \ s(e)=v, \ r(e) = w \ \}|.
$$
Conversely, we have seen that given a matrix $A$ indexed by $V$ with cardinal entries, one may form a directed graph $G_A=(V,E_A,r,s)$ where $E_A$ is the set of triples $(v,w,\alpha)$ such that $v,w\in V$, $0\leq \alpha < A_{vw}$ is an ordinal, $A_{vw}$ is interpreted as the least ordinal with cardinality $A_{vw}$, while $r(v,w,\alpha) = w$ and $s(v,w,\alpha) = v$. It is clear that $G_{A_E}$ and $G$ are isomorphic directed graphs and that $A_{E_A} = A$. 

The following shows that under countability / finiteness assumptions, shift equivalence with arbitrary cardinals becomes the standard notion we know from the literature.

\begin{Prop}
Let $A$ and $B$ be matrices with cardinal entries, indexed by sets $V$ and $W$. Suppose that $V$ and $W$ are countable / finite, and suppose that $A$ and $B$ are over $\mathbb{N} \cup \{\aleph_0 \}$ / over $\mathbb{N}$, respectively. If $A$ and $B$ are shift equivalent, then the matrices $R$ and $S$ realizing shift equivalence can be chosen to be with entries in $\mathbb{N} \cup \{\aleph_0 \}$ / $\mathbb{N}$, respectively.
\end{Prop}

\begin{proof}
Suppose now that $V$ and $W$ are countable / finite, and that $A$ and $B$ are with entries in $\mathbb{N} \cup \{\aleph_0 \}$ / $\mathbb{N}$ respectively. If $R$ and $S$ implement shift equivalence of lag $m$ between $A$ and $B$, denote by $R'$ and $S'$ the matrices obtained from $R$ and $S$ by replacing all non-countable / non-finite entries with zeros (respectively). As the matrices $A^m$ and $B^m$ are with entries in $\mathbb{N} \cup \{\aleph_0 \}$ / $\mathbb{N}$ (respectively), we still have that $R'S' = A^m$ and $S'R' = B^m$, as well as $S'A=BS'$ and $AR' = R'B$. Hence, $A$ and $B$ are shift equivalent via $R'$ and $S'$, and we are done.
\end{proof}

In this paper we will conduct our study through the lens of graph C*-algebras, which include the class of Cuntz-Krieger C*-algebras. We recommend \cite{FLR00, Rae05} and the references therein for more on graph C*-algebras. We will sometime assume that our graphs have finite out-degree, which is often called ``row-finiteness" in the literature.

\begin{Def}
Let $G=(V,E,r,s)$ be a directed graph. A family of operators $(S_v,S_e)_{v\in V, e\in E}$ on a Hilbert space $\mathcal{H}$ is called a \emph{Cuntz-Krieger family} if
\begin{enumerate}
\item $(S_v)_{v\in V}$ is a family of pairwise orthogonal projections;

\item $S_e^*S_e = S_{r(e)}$ for all $e\in E$, and;

\item $\sum_{e\in s^{-1}(v)} S_eS_e^* = S_v$ for all $v\in V$ with $0 < |s^{-1}(v)| < \infty$.
\end{enumerate}
The \emph{graph C*-algebra} $C^*(G)$ of $G$ is the universal C*-algebra generated by a Cuntz-Krieger families.
\end{Def} 

Universality of $C^*(G)$ gives rise to a point-norm continuous \emph{gauge action} of the unit circle $\gamma : \mathbb{T} \rightarrow \Aut(C^*(G))$ given by
$$
\gamma_z(S_v) = S_v, \text{ and } \gamma_z(S_e) = z\cdot S_e, \text{ for} \ z\in \mathbb{T},\ v\in V,\ e\in E.
$$
With this gauge action, $C^*(G)$ becomes a $\mathbb{Z}$-graded C*-algebra whose graded components are $C^*(G)_n:= \{ \ \ T \in C^*(G) \ \ | \ \ \gamma_z(T) = z^n \cdot T \ \ \}$.

We will sometime assume that our graphs have finite out-degree, which is often called ``row-finiteness" in the literature. We shall need the groupoid C*-algebra description of $C^*(G)$, as specified in \cite{CR17}. Indeed, if $G=(V,E)$ is a directed graph with finite out-degree, no sources and no sinks, we may construct the locally compact, Hausdorff etale groupoid
$$
\G_E:= \{ \ (x,m-n,y) \in X_E^+ \times \mathbb{Z} \times X_E^+ \ | \ x,y \in X_E, \ m,n \in \bbN, \ \sigma_E^m(x) = \sigma_E^n(y) \ \}
$$
with product $(x,k,y)(w,\ell, z) = (x, k+\ell, z)$ if $y=w$ (otherwise undefined), and inverse $(x,k,y)^{-1} = (y,-k,x)$. The topology on $\G_E$ is generated by subsets of the form
$$
Z(U,m,n,V) = \{ \ (x,m-n,y) \in \G_E \ | \ x\in U, \ y\in V, \ \}
$$
where $m,n \in \mathbb{N}$ and $U, V \subseteq X_E^+$ clopen such that $\sigma_E^m|_U$ is injective, $\sigma_E^n|_V$ is injective, and $\sigma_E^m(U) = \sigma_E^n(V)$. The map $x \mapsto (x,0,x)$ then provides a homeomorphism from $X_E^+$ into the unit space $\G_E^0$ of $\G_E$. It is well-known that $C^*(G) \cong C^*(\G_E)$ is the groupoid C*-algebra of $\G_E$, and that $ \mathcal{D}_E:= C_0(\G_E^{(0)}) \cong C_0(X_E^+)$ is the subalgebra of continuous functions on units of $\G_E$. This subalgebra is often called the \emph{diagonal subalgebra} of $C^*(G)$, and is given by
$$
\mathcal{D}_E = \overline{\Span}\{ \ S_{\mu} S_{\mu}^* \ | \  \mu \in E^n, \ n \in \mathbb{N} \ \}.
$$

In what follows, recall that $\bbK$ denotes compact operators on $\ell^2(\mathbb{N})$, which contains a natural copy of diagonal compact operators $c_0 \subseteq \bbK$. We will consider $\gamma \otimes \id_{\bbK}$ as the standard gauge action on the stabilization $C^*(G) \otimes \bbK$. It was shown in \cite[Theorem 5.1]{CR17} that for any two finite graphs without sources and sinks $G=(V,E)$ and $G'=(V',E')$ we have that $(X_E,\sigma_E)$ and $(X_{E'},\sigma_{E'})$ are conjugate if and only if there is an equivariant isomorphism $\varphi: C^*(G) \otimes \bbK \rightarrow C^*(G') \otimes \bbK$ such that $\varphi(\mathcal{D}_E \otimes c_0) = \mathcal{D}_{E'} \otimes c_0$.

Suppose that $A=(A_{vw})_{v,w=0}^{n-1}$ is an essential $n\times n$ matrix with non-negative integer entries. The \emph{Cuntz--Krieger algebra} of $A$ is the universal C*-algebra $\O_A$ generated by a family $\{S_{(v,w,m)}\mid v,w,m\in\bbN,\ 0\le v,w < n,\ 0 \le m < A_{vw} \}$ of partial isometries satisfying
\begin{enumerate}
	\item $S_{(v,w,m)}^*S_{(v,w,m)}=\sum_{u=0}^{n-1}\sum_{\ell=0}^{A_{wu}-1}S_{(w,u,\ell)}S_{(w,u,\ell)}^*$,
	\item $\sum^{n-1}_{v,w=0}\sum_{m=0}^{A_{vw}-1} S_{(v,w,m)}S_{(v,w,m)}^*=1$.
\end{enumerate}
For each family as above, we get projections $S_w = S_{(v,w,m)}^*S_{(v,w,m)}$ (independent of $v$ and $m$), so that the family $(S_v, S_{(v',w',m)})$ becomes a Cuntz-Krieger family for the directed graph $G_A$ (see \cite[Remark 2.16]{CK80} and \cite[Remark 2.8]{Rae05}). Hence, there is a $*$-isomorphism between $\O_A$ and $C^*(G_A)$ which maps the generators $S_{(v,w,m)}$ of $\O_A$ to edge generators for $C^*(G_A)$. This isomorphism between $\O_A$ and $C^*(G_A)$ induces the usual gauge action on $\O_A$ from \cite{CK80} (which is also built from universality of $\O_A$, see \cite[Remark 2.2(2)]{HR97}) and sends the natural diagonal subalgebra of $\O_A$ to the diagonal subalgebra $\mathcal{D}_E$ inside $C^*(G_A)$. Thus, there is no loss of generality arising from considering graph C*-algebras instead of Cuntz-Kriger C*-algebras.

\begin{Not}
Whenever $X, Y \subseteq \mathcal{L}(E)$ are norm-closed subspaces, we denote by $XY$ or by $X\cdot Y$ the closed linear span of products $x\cdot y$ with $x \in X$ and $y\in Y$.
\end{Not}

\section{Shift equivalence and Cuntz-Pimsner algebras} \label{s:SE-CP}

In this section we discuss C*-correspondences of adjacency matrices, Cuntz-Pimsner C*-algebras and Pimsner dilations. The main result of this section is a characterization of equivariant stable isomorphism of Cuntz-Pimsner algebras in terms of Pimsner dilations.

We will need some of the theory of C*-correspondences. We mention some of the basic definitions, but will assume some familiarity with the theory of Hilbert C*-modules as in \cite{Lan95,MT05}.
	
\begin{Def}\label{d: correspondence} A \textit{C*-correspondence} from $\A$ to $\B$ (or $\B-\A$ C*-corres\-pondence) is a right Hilbert $\A$-module $E$ with a *-representation $\phi_E: \B \to \L(E)$, where $\L(E)$ is the C*-algebra of adjointable operators on $E$. We denote by $\K(E)$ the ideal of $\L(E)$ which is closed linear span of rank one operators $\theta_{\xi,\eta}$ given by $\theta_{\xi,\eta}(\zeta) = \xi \langle \eta,\zeta \rangle$.

We will say that two $\B-\A$ correspondences $E$ and $F$ are \emph{unitarily isomorphic} (and denote this $E \cong F$) if there is an isometric surjection $U :E \rightarrow F$ such that for every $a\in \A$, $b\in \B$ and $\xi \in E$ we have
$$
U(\phi_E(b) \xi a) = \phi_F(b) U(\xi) a.
$$
\end{Def}

 We will assume throughout this note that C*-correspondences $E$ are non-degenerate (sometimes called essential) in the sense that $\phi_E(\B) E = E$. We say that a $\B-\A$ C*-corres\-pondence $E$ is \emph{regular} if its left action $\phi_E$ is injective and $\phi_E(\B) \subseteq \K(E)$. When the context is clear, we write $b \xi$ to mean $\phi_E(b) \xi$ for $b\in \B$ and $\xi \in E$. Finally, we say that a $\B-\A$ C*-correspondence $E$ is \emph{full} if $\A$ is equal to the ideal $\langle E, E \rangle$ defined as the closed linear span of $\langle \xi,\eta \rangle$ for $\xi,\eta \in E$.

The most important examples of C*-correspondences in our study are the ones coming from adjacency matrices. Let $V$ and $W$ be sets, and let $C$ be a $V \times W$ matrix so that $C_{vw}$ is some cardinal. We denote by 
$$
E_C: = \{ \ (v,w,\alpha) \ | \ 0\leq \alpha < C_{vw}, \ v\in V, \ w\in W \ \}.
$$

When $C$ is a $V \times W$ matrix with cardinal entries, we may construct a $c_0(V)-c_0(W)$ correspondence $X(C)$ by taking the Hausdorff completion of all finitely supported functions on $E_C$ with respect to the inner product
$$
\langle \xi, \eta \rangle(w) = \sum_{(v,w,\alpha)\in E_C} \overline{\xi(v,w,\alpha)} \eta(v,w,\alpha),
$$
where $\xi$ and $\eta$ are finitely supported on $E_C$. The left actions of $c_0(V)$ and the right action of $c_0(W)$ on $X(C)$ are then given by 
$$(f\cdot \xi \cdot g)(v,w,\alpha) = f(v)\xi(v,w,\alpha)g(w)\text{ for }f\in c_0(V) \text{ and }g\in c_0(W).$$ 
When $V=W$, it is clear that $X(C)$ coincides with the graph C*-cor\-re\-spon\-den\-ce $X(G_C)$ of the directed $G_C=(V,E_C)$ as is explained in the discussion preceding \cite[Theorem 6.2]{DEG2020}, with range and source interchanged in the definitions of inner product and bimodule actions.

We will need to know that in the above way we obtain all $c_0(V)-c_0(W)$ correspondences. This type of result was first shown by Kaliszewski, Patani and Quigg \cite{KPQ12} when $C$ is a square matrix indexed by a countable set and with countable entries.

\begin{Prop} \label{p:adjacency-corresp}
Let $V$ and $W$ be sets, and $E$ a $c_0(V)-c_0(W)$ correspon\-dence. Then there exists a unique $V\times W$ matrix $C$ with cardinal entries such that $E$ is unitarily isomorphic to $X(C)$.
\end{Prop}

\begin{proof}
Let $v\in V$ and $w\in W$, and let $p_v \in c_0(V)$ and $p_w \in c_0(W)$ be the characteristic functions of $\{v\}$ and $\{w\}$ respectively. Since $p_v E p_w$ is a Hilbert space, we let $C_{vw}$ be its dimension, and let $\{e_{(v,w,\alpha)} \}_{0 \leq \alpha< C_{vw}}$ be an orthonormal basis for it, indexed by ordinals $0 \leq \alpha < C_{vw}$. Then clearly $C$ is a $V \times W$ matrix with cardinal values. We define a map $U : X(C) \rightarrow E$ on finitely supported functions by setting $U(\xi) = \sum_{(v,w,\alpha)\in E_C} \xi(v,w,\alpha)e_{(v,w,\alpha)}$ for finitely supported $\xi \in X(C)$. Now, for a finitely supported function $\xi \in X(C)$ we have for fixed $w\in W$ that
$$
|U(\xi)|^2(w) = \sum_{(v,w,\alpha)\in E_C} |\xi(v,w,\alpha)|^2 = | \xi|^2(w),
$$
Hence, $U$ extends to an isometry on $X(C)$, and since the linear span of $\{ \ e_{(v,w,\alpha)} \ | \ v\in V, \ w\in W \ 0 \leq \alpha < C_{vw} \ \}$ is dense in $E$, we get that $U$ is a unitary isomorphism.

As for uniqueness, suppose $F$ is another $c_0(W)-c_0(V)$-correspondence which is unitarily isomorphic to $E$ via a unitary $U$. Hence, if $C^E$ and $C^F$ are the matrices associated to $E$ and $F$ via the first paragraph, we must have $C^F_{vw} = C^E_{vw}$, and we are done.
\end{proof}

For a $\B-\A$ correspondence $E$ and a $\C-\B$ correspondence $F$, we can form the \emph{interior tensor product} $\C - \A$  correspondence $F\otimes_{\B}E$ of $E$ and $F$ as follows. Let $F\otimes_{\alg} E$ denote the quotient of the algebraic tensor product, by the subspace generated by elements of the form:
	$$
	\eta b \otimes \xi - \eta \otimes b\xi, \ \ \text{for} \ \  \xi \in E, \eta \in F, b\in \B.
	$$
	Define an $\A$-valued semi-inner product and left $\C$-action by setting:
	$$ 
	\langle \eta_1\otimes \xi_1, \eta_2\otimes \xi_2\rangle = \langle \xi_1, \langle \eta_1,\eta_2\rangle \cdot \xi_2\rangle, \ \text{for} \ \ \xi_1,\xi_2\in E, \ \eta_1, \eta_2\in F
	$$
	$$
	c \cdot(\eta \otimes \xi) = (c \cdot \eta )\otimes \xi, \ \ \text{for} \ \ \eta \in E, \ \xi \in F, \ c\in \C.
	$$
We denote by $F\otimes_{\B} E$ the separated completion of $F\otimes_{\alg} E$ with respect to the $\A$-valued semi-inner product defined above. One then verifies that $F\otimes_{\B} E$ is a $\C-\A$ correspondence (see for example \cite[Proposition 4.5]{Lan95}). We will often abuse notation and write $F \otimes E$ for $F \otimes_{\B} E$ when the context is clear. If $E$ is an $\A-\A$ correspondence (a C*-correspondence over $\A$), then we denote by $E^{\otimes n}$ the $n$-fold interior tensor product of $E$ with itself.

\begin{Prop} \label{p:multip-adjacency}
Let $I,J,K$ be sets, and let $C$ be an $I\times J$ matrix and $D$ be a $J \times K$ matrix, both with cardinal entries. Then there is a unitary isomorphism $U:X(C)\otimes X(D) \rightarrow X(CD)$.
\end{Prop}

\begin{proof}
For each $i \in I$, $j\in J$ and $k\in K$, let $C_{ij} = |X_{ij}|$, $D_{jk} = |Y_{jk}|$ for some sets $X_{ij}$ and $Y_{jk}$. Then $(CD)_{ik}$ is the cardinality of the disjoint union of Cartesian products $\sqcup_{j\in J} X_{ij}\times Y_{jk}$.

Now, clearly both $X(C)\otimes X(D)$ and $X(CD)$ are $c_0(I)-c_0(K)$ correspondences, so that by the uniqueness part of Proposition \ref{p:adjacency-corresp} it will suffice to show for every $i\in I$ and $k\in K$ that the dimension of the Hilbert space $p_i X(C)\otimes X(D)p_k$ is equal to $(CD)_{ik}$.

So let $\{e_x\}_{x \in X_{ij}}$ be an orthonormal basis for $p_i X(C)p_j$ for $i\in I$ and $j\in J$ and let $\{e_y \}_{y \in Y_{j'k}}$ be an orthonormal basis for $p_{j'}X(D)p_k$ for $j'\in J$ and $k\in K$. Then, for $x\in X_{ij}$ and $y\in Y_{j'k}$ we have $e_x \otimes e_y \neq 0$ if and only if $j=j'$. Hence, an orthonormal basis for $p_i X(C)\otimes X(D)p_k$ is given by $\{ \ e_x \otimes e_y \ | \ x \in X_{ij}, \ y \in Y_{jk}, \ j\in J \ \}$. The cardinality of this basis is clearly equal to $(CD)_{ik}$, so the proof is concluded.
\end{proof}

The following definition of shift equivalence of C*-correspondences first appeared in \cite{KK14}.

\begin{Def} \label{d:sme-se}
Let $E$ and $F$ be C*-correspondences over $\A$ and $\B$ respectively. We say that $E$ and $F$ are 
\begin{enumerate}
\item \emph{shift equivalent} with lag $m$ if there are $m\in \bN \setminus \{0 \}$, an $\A-\B$ correspondence $R$ and a $\B-\A$ -correspondence $S$, together with unitary isomorphisms
\begin{equation*}
E^{\otimes m} \cong R \otimes S, \ \ F^{\otimes m} \cong S \otimes R,
\end{equation*}
\begin{equation*}
S \otimes E \cong F \otimes S, \ \ E \otimes R \cong R \otimes F.
\end{equation*}

\item \emph{elementary shift related} if they are shift equivalent with lag $1$
\item \emph{strong shift equivalent} if they are equivalent in the transitive closure of the elementary shift relation.

\end{enumerate}
\end{Def}

The following shows that shift equivalence between two C*-corres\-pondences generalizes shift equi\-valence of adjacency matrices. 

\begin{Prop} \label{p:shift-equiv-generalizes}
Let $A$ and $B$ be matrices with cardinal entries, indexed by sets $V$ and $W$. Then, $A$ and $B$ are shift equivalent if and only if $X(A)$ and $X(B)$ are shift equivalent.
\end{Prop}

\begin{proof}
By Proposition \ref{p:multip-adjacency} we see that if $A$ and $B$ are shift equivalent with lag $m$ via matrices $C$ and $D$, then the C*-correspondences $X(A)$ and $X(B)$ are shift equivalent with lag $m$ via the C*-correspondences $X(C)$ and $X(D)$.

Conversely, suppose $R$ and $S$ are $c_0(V)-c_0(W)$ and $c_0(W)-c_0(V)$ correspondences implementing shift equivalence of lag $m$ for $X(A)$ and $X(B)$. By Proposition \ref{p:adjacency-corresp} there are a $V\times W$ matrix $C$ and $W\times V$ matrix $D$ with unitary isomorphisms $R \cong X(C)$ and $S \cong X(D)$. Hence, the uniqueness portion in Proposition \ref{p:adjacency-corresp} guarantees that $C$ and $D$ implement a shift equivalence of $A$ and $B$ with lag $m$.
\end{proof}

Every path isomorphism $\phi : E_A \rightarrow E_B$ induces a unitary isomorphism $\Phi : X(A) \rightarrow X(B)$ by setting $\Phi(\xi)(v,w,\alpha) = \xi(\phi^{-1}(v,w,\alpha))$. However, the converse is in general false; it is easy to construct a unitary isomorphism $U: X(A) \rightarrow X(B)$ for which there is no path isomorphism $\phi : E_A \rightarrow E_B$ such that $U$ is the induced unitary isomorphism from $\phi$. The point of Proposition \ref{p:shift-equiv-generalizes} is that we only need to find \emph{some} path isomorphisms for each one of the four relations appearing in the definition of shift equivlaence, and not necessarily path isomorphisms which induce \emph{the same} four isomorphisms we started with at the level of C*-correspondences.

\begin{Rmk}
From considerations similar to the above we see that strong shift equivalence of matrices implies the strong shift equivalence of their associated C*-correspondences. The converse, however, is unknown.
\end{Rmk}

Next we discuss Cuntz-Pimsner algebras and Pimsner dilations.
More material on Cuntz-Pimsner algebras, with a special emphasis on C*-corres\-pondences of graphs, can be found in \cite[Chapter 8]{Rae05}. We note immediately that what we call a \emph{rigged} representation here is often referred to as an \emph{isometric} representation in the literature (see \cite{MS00}).
	
	\begin{Def}
		Let $E$ be a C*-correspondence from $\A$ to $\B$, and $\C$ some C*-algebra. A \emph{rigged representation} of $E$ is a triple $(\pi_{\mathcal{A}},\pi_{\mathcal{B}}, t)$ such that $\pi_{\mathcal{A}} : \A \rightarrow \C$ and $\pi_{\mathcal{B}} : \B \rightarrow \C$ are *-homomorphisms and $t : E \rightarrow \C$ is a linear map such that
		\begin{enumerate}
			\item
			$\pi_{\mathcal{B}}(b)t(\xi) \pi_{\mathcal{A}}(a) = t(b \cdot \xi \cdot a)$ for $a\in \A$, $b\in \B$ and $\xi \in E$.
			\item $t(\xi)^*t(\eta) = \pi_{\mathcal{A}}(\langle \xi,\eta \rangle)$
		\end{enumerate}
		We say that $(\pi_{\mathcal{A}},\pi_{\mathcal{B}}, t)$ is \textit{injective} if both $\pi_{\mathcal{A}}$ and $\pi_{\mathcal{B}}$ are injective *-homo\-morphisms.
		We denote by $C^*(\pi_{\mathcal{A}},\pi_{\mathcal{B}},t)$ the C*-algebra generated by the images of $\pi_{\mathcal{A}}$,$\pi_{\mathcal{B}}$ and $t$.
	\end{Def}
	
We will mostly be concerned with the situation where $\A=\B$ and $\pi_{\mathcal{A}} = \pi_{\mathcal{B}}$, in which case we will denote $\pi:=\pi_{\mathcal{A}} = \pi_{\mathcal{B}}$, and refer to the representation as a pair $(\pi,t)$, and the generated C*-algebra by $C^*(\pi,t)$.  

Now let $E$ be a C*-correspondence over $\A$. The Toeplitz-Pimsner algebra $\T(E)$ is then the universal C*-algebra generated by a rigged representation of $E$. Universality of $\T(E)$ implies that it comes equipped with a point-norm continuous gauge action $\gamma: \bT \to \Aut(\T(E))$ given by
	$$
	\gamma_z(\pi(a))= \pi(a), \ \text{ and } \ \  \gamma_z(t(\xi))=z \cdot t(\xi), \ \ \text{for} \ z\in \bT, \ \xi\in E, \ a\in \A.
	$$

	Toeplitz-Pimsner algebras have a canonical quotient, also known as the Cuntz-Pimsner algebra, which was originally defined by Pimsner in \cite{Pim95} and refined by Katsura in \cite{Kat04b}.
	
	\begin{Def}\label{d:Katsura-ideal}
		For a C*-correspondence $E$ over $\A$, we define \emph{Katsura's ideal} $J_E$ in $\A$ by
		$$J_E\coloneqq  \{a\in \A: \phi_E(a)\in \K(E) \text{ and } ab=0 \text{ for all } b\in \ker \phi_E \}.$$
	\end{Def}
	
	For a rigged representation $(\pi,t)$ of a C*-correspondence $E$ over $\A$, it is known there is a well-defined *-homomorphism $\psi_t: \K(E)\to C^*(\pi,t)$ given by $\psi_t(\theta_{\xi,\eta})=t(\xi)t(\eta)^*$ for $\xi,\eta\in E$ (see for instance \cite[Lemma 2.2]{KPW98}).
	
	\begin{Def}\label{d:cov-rep}
		A rigged representation $(\pi,t)$ is said to be \textit{covariant} if $\pi(a)=\psi_t(\phi_E(a))$, for all $a\in J_E$.
	\end{Def}
	
	The Cuntz-Pimsner algebra $\O(E)$ is then the universal C*-algebra generated by a \textit{covariant} representation of $E$. Suppose now that $(\pi,t)$ and $(\overline{\pi},\overline{t})$ are universal rigged and covariant representations respectively, so that $\T(E) = C^*(\pi,t)$ and $\O(E) = C^*(\overline{\pi},\overline{t})$. We denote by $\J_E$ the kernel of the natural quotient map from $\T(E)$ onto $\O(E)$, which is the ideal generated by elements of the form $\pi(a) - \psi_t(\phi_E(a))$ for $a\in J_E$. Since the ideal $\J_E$ is gauge invariant, we see there is an induced point-norm continuous circle action $\gamma : \bT \rightarrow \Aut(\O(E))$ given by
	$$
	\gamma_z(\pi(a))= \pi(a), \ \text{ and } \ \  \gamma_z(t(\xi))=z \cdot t(\xi), \ \ \text{for} \ z\in \bT, \ \xi\in E, \ a\in \A.
	$$
	
It follows that $\O(E)$ then becomes a $\mathbb{Z}$-graded C*-algebra, with its $n$-th graded component given by $\O(E)_n:= \{ \ c \in \O(E) \ | \ \gamma_z(c) = z^n \cdot c \ \}$. By \cite[Theorem 3]{Rae18} we see that there is a bijective correspondence between topologically $\mathbb{Z}$-graded $C^*$-algebras, with graded *-homomorphisms and $C^*$-algebras equipped with a circle action, together with equivariant *-homomorphisms. In particular, an isomorphism $\varphi : \O(E) \rightarrow \O(F)$ between two Cuntz-Pimsner algebras is equivariant if and only if it is graded. 

\begin{Exl}
When $G=(V,E)$ is a directed graph we know that the graph C*-algebra $C^*(G)$ coincides with the Cuntz-Pimsner algebra of the correspondence $X(A_E)$ by an isomorphism that intertwines the gauge action of $C^*(G)$ and the gauge action of $\O(X(A_E))$. See \cite[Section 8]{Rae05} for more details.
\end{Exl}

\begin{Def} \label{d:isomorphisms}
		Let $E, F$ be C*-correspondences over $\A$ and $\B$, respectively. Denote by $\gamma^E$ and $\gamma^F$ the gauge actions on $\O(E)$ and $\O(F)$ respectively. We say that $\O(E)$ and $\O(F)$ are \emph{equivariantly stably isomorphic} if there is a $*$-isomorphism $\varphi : \O(E) \otimes \bbK \rightarrow \O(F) \otimes \bbK$ such that $\varphi \circ (\gamma^E_z \otimes \id) = (\gamma^F_z \otimes \id) \circ \varphi$ for every $z\in \bT$.
	\end{Def}
	
Equivariant stable isomorphisms arise naturally in the classification of groupoid C*-algebras, but always in a way which respects diagonal subalgebras. For instance in \cite{CRST17}, equivariant stable isomorphisms which respect diagonal subalgebras are characterized in terms of isomorphisms of graded groupoids. We will get back to such isomorphisms in Section \ref{s:smse}.

The following definition is similar to Definition \ref{d: correspondence}, and we will provide a precise distinction between the two in the remark that follows.

\begin{Def} \label{d:unitary-iso-dif-alg}
Suppose $E$ and $F$ are C*-correspondences over C*-algeb\-ras $\A$ and $\B$ respectively. We say that $E$ and $F$ are \textit{unitarily isomorphic} (denoted $E\cong F$) if there exist a surjective, isometric map $U: E\to F$ and a *-isomorphism $\rho: \A\to \B$, s.t. $U(b \cdot \xi\cdot a)=\rho(b)\cdot U(\xi)\cdot \rho(a)$ for all $a,b\in \A, \xi \in E$.
\end{Def}

\begin{Rmk}
Suppose now that $E$ and $F$ are C*-correspondences over C*-algebras $\A$ and $\B$ respectively, and that $U : E \rightarrow F$ is a unitary isomorphism implemented by a *-isomorphism $\rho : \A \rightarrow \B$. By the discussion in \cite[Subsection 2.1]{DO18} we may ``twist" the C*-correspondence $F$ to a C*-correspondence $F_{\rho}$ over $\A$ so that $U : E \rightarrow F_{\rho}$ becomes a unitary isomorphism as in Definition \ref{d: correspondence}. More precisely, the new operations on $F_{\rho}$ are given by $ {\langle \xi,\eta \rangle}_{\rho} \coloneqq \rho^{-1}({\langle \xi,\eta \rangle}_\B), \text { for } \xi,\eta\in F; \ a \cdot \xi= \rho(a)\cdot \xi; \text{ and } \xi\cdot a \coloneqq \xi \cdot \rho(a), \text{ for all } \xi\in F \text{ and } a\in \A$, and the identity map $\id_F^{\rho} :F \to F_{\rho}$ becomes a unitary isomorphism as in Definition \ref{d:unitary-iso-dif-alg}. Then, the isometric surjection $\id_F^{\rho} \circ U$ is a unitary isomorphism as in Definition \ref{d: correspondence}. Hence, we can go back and forth between the two definitions of unitary isomorphism. 

We warn the reader that unitary isomorphism as in Definition \ref{d:unitary-iso-dif-alg} is not the same as having an isometric surjection $U$ implemented via two potentially \emph{different} *-isomorphisms $\rho_1 : \A \rightarrow \B$ and $\rho_2 : \A \rightarrow \B$ in the sense that $U(b \cdot \xi \cdot a) = \rho_1(a) U(\xi) \rho_2(b)$ for $a,b\in \A$. We also note that whenever one of our C*-correspondences has possibly different left and right coefficient C*-algebras, we only consider one notion of unitary isomorphism, which is the one in Definition \ref{d: correspondence}.
\end{Rmk}

In what follows, we say that $M$ is an \emph{imprimitivity} $\A-\B$ correspondence (from $\B$ to $\A$) if it is full and its left action $\phi_M$ is a $*$-isomorphism onto $\K(E)$. The following was introduced by Muhly and Solel \cite{MS00} in their study of tensor algebras of C*-correspondences.

\begin{Def}
Let $E$ and $F$ be C*-correspondences over C*-algebras $\A$ and $\B$ respectively. We say that $E$ and $F$ are \emph{strongly Morita equivalent} if there are an imprimitivity $\A-\B$ bimodule $M$ and an isometric surjective linear map $ U: E\otimes M \rightarrow M \otimes F$ such that for every $\xi \in E \otimes M$, $a\in \A$ and $b\in \B$ we have $U(a\xi b) = aU(\xi) b$.
\end{Def}

When $M$ is an imprimitivity $\A-\B$-bimodule, there is an ``inverse" imprimitivity $\B-\A$-bimodule $M^*$ satisfying $M \otimes M^* \cong \A$ and $M^* \otimes M \cong \B$. If moreover $\A = \B$, it follows from \cite[Theorem 2.9]{AEE98} that the $\bZ$-graded components of $\O(M)$ are given for $n\in \mathbb{N}$ by
\begin{itemize}
\item $\O(M)_n \cong M^{\otimes n}$ for $n> 0$,
\item $\O(M)_0 \cong \A$, and
\item $\O(M)_n \cong (M^{\otimes n})^*$ for $n < 0$.
\end{itemize}

\begin{Prop} \label{P:stable-equivariant}
Let $E$ and $F$ be C*-correspondences over C*-algebras $\A$ and $\B$ respectively. If $E$ and $F$ are unitarily equivalent, then $\O(E)$ and $\O(F)$ are equivariantly isomorphic. If moreover $E$ and $F$ are imprimitivity bimodules, then $E$ and $F$ are unitarily equivalent if and only if $\O(E)$ and $\O(F)$ are equivariantly isomorphic.
\end{Prop}

\begin{proof}
Let $(\overline{\pi},\overline{t})$ be a universal covariant representation for $\O(E)$. Assume that $E$ and $F$ are unitarily equivalent. By the implication $(1) \implies (4)$ of \cite[Corollary 3.5]{DEG2020} we get that $\T(E)$ and $\T(F)$ are graded isomorphic, and hence equivariantly isomorphic. Then, \cite[Theorem 3.1]{DEG2020} gives rise to an induced isomorphism between $\O(E)$ and $\O(F)$, which is actually equivariant since the ideals $\J_E$ and $\J_F$ are gauge invariant.

Conversely, if $E$ and $F$ are imprimitivity bimodules, and $\varphi : \O(E) \rightarrow \O(F)$ is an equivariant isomorphism. Then $\varphi$ is $\mathbb{Z}$-graded, and we get that the restriction $U:= \varphi|_{\O(E)_1} : \O(E)_1 \rightarrow \O(F)_1$ is an isometric $\O(E)_0 - \O(F)_0$ bimodule isomorphism, implemented by the $*$-isomorphism $\rho:= \varphi|_{\O(E)_0} : \O(E)_0 \rightarrow \O(F)_0$. From the identifications in the discussion preceding the theorem, we get that $E$ and $F$ are unitarily isomorphic. 
\end{proof}

Given a C*-correspondence $E$ over $\A$, we may form the external minimal tensor product $E \otimes \mathbb{K}$, which is a C*-correspondence over $\A \otimes \mathbb{K}$ as defined in \cite[p. 34]{Lan95}. A consequence of \cite[Proposition 2.10]{DEG2020} is that $\O(E\otimes \mathbb{K})$ is canonically isomorphic to $\O(E)\otimes \mathbb{K}$ via a map induced by the representation $(\overline{\pi}\otimes \id, \overline{t} \otimes \id)$.

\begin{Cor} \label{C:equivariant-stable-iso}
Let $E$ and $F$ be C*-correspondences over C*-algebras $\A$ and $\B$ respectively. If $E\otimes \bbK$ and $F\otimes \bbK$ are unitarily equivalent, then $\O(E)$ and $\O(F)$ are equivariantly stably isomorphic. If moreover $E$ and $F$ are \emph{imprimitivity bimodules}, then $E\otimes \bbK$ and $F\otimes \bbK$ are unitarily equivalent if and only if $\O(E)$ and $\O(F)$ are equivariantly stably isomorphic.
\end{Cor}

\begin{proof}
Since $E\otimes \mathbb{K}$ and $F\otimes \mathbb{K}$ are unitarily equivalent, by Proposition \ref{P:stable-equivariant} we get that $\O(E\otimes \mathbb{K})\cong \O(F\otimes \mathbb{K})$ equivariantly.

Suppose now that $(\overline{\pi}, \overline{t})$ is a universal covariant representation of $E$. By \cite[Proposition 2.10]{DEG2020}, we get that the representation $(\overline{\pi}\otimes \id, \overline{t} \otimes \id)$ induces an isomorphism $\overline{\rho} : \O(E\otimes \mathbb{K}) \rightarrow \O(E)\otimes \mathbb{K}$ satisfying $\overline{\rho}(\xi \otimes K) = \overline{t}(\xi) \otimes K$ for every $\xi \in E$ and $K\in \mathbb{K}$. Hence $\overline{\rho}$ must be equivariant. It is similarly shown that $\O(F \otimes \mathbb{K}) \cong \O(F)\otimes \mathbb{K}$ equivariantly. Hence, we get equivariantly that
$$
\O(E) \otimes \mathbb{K} \cong \O(E\otimes \mathbb{K}) \cong \O(F \otimes \mathbb{K}) \cong \O(F) \otimes \mathbb{K}.
$$
With this, we obtain the first part of our result.

Conversely, if moreover $E$ and $F$ are imprimitivity bimodules, then so are $E\otimes \mathbb{K}$ and $F \otimes \mathbb{K}$. Thus, we are done by Proposition \ref{P:stable-equivariant} and the above identifications. 
\end{proof}

Let $E$ be a C*-correspondence over a C*-algebra $\A$, and $(\overline{\pi},\overline{t})$ a universal covariant representation of $E$, we denote by 
$$
\A_{\infty} = \O(E)_0 = \overline{\Span} \{ \ \overline{\pi}(\A), \ \psi_{\overline{t}^m}(\K(E^{\otimes m})) \ | \ m\geq 1 \ \}
$$ 
the fixed point algebra of $\O(E)$ under the gauge action $\gamma$. It is well-known that $\A_{\infty}$ is the direct limit of C*-subalgebras $\A_n$ given by 
$$
\A_n = \overline{\Span} \{ \ \overline{\pi}(\A), \ \psi_{\overline{t}^m}(\K(E^{\otimes m})) \ | \ 1 \leq m \leq n \ \}.
$$
When $E$ has an injective left action, we get that $\A_n = \psi_{\overline{t}^n} (\K(E^{\otimes m}))$ for each $n\in \mathbb{N}$, and we define $E_n : = \overline{t}(E) \A_n$ inside $\O(E)$. Then, $E_n$ becomes a C*-correspondence from $\A_n$ to $\A_{n+1}$, where the left action of $\A_{n+1}$ is defined by left multiplication in $\O(E)$. By taking the direct limit $E_{\infty} := \overline{t}(E) \cdot \A_{\infty} = \overline{\cup_{n=1}^{\infty} E_n}$, we obtain a C*-correspondence over $\A_{\infty}$ called the \emph{Pimsner dilation} of $E$. 

It was shown by Pimsner in \cite[Theorem 2.5 (2)]{Pim95} that the identification $E \otimes \A_{\infty} \cong E_{\infty}$ gives rise to an equivariant $*$-injection $\T(E) \rightarrow \T(E_{\infty})$, which then induces an equivariant isomorphism $\O(E) \cong \O(E_{\infty})$ between the quotients. When $E$ is also \emph{regular} and \emph{full}, the C*-correspondence $E_{\infty}$ becomes an imprimitivity bimodule.

\begin{Thm} \label{t:down-and-up}
Let $E$ and $F$ be \emph{regular} and \emph{full} C*-correspon\-dences over $\sigma$-unital C*-algebras $\A$ and $\B$ respectively. Then, $\O(E)$ and $\O(F)$ are equivariantly stably isomorphic if and only if $E_{\infty}$ and $F_{\infty}$ are strong Morita equivalent.
\end{Thm}

\begin{proof}
Since $E$ and $F$ are regular and full, we get that $E_{\infty}$ and $F_{\infty}$ are imprimitivity bimodules. Hence, by Corollary \ref{C:equivariant-stable-iso} we get that $\O(E)$ and $\O(F)$ are equivariantly stably isomorphic if and only if $E_{\infty} \otimes \mathbb{K}$ and $F_{\infty} \otimes \mathbb{K}$ are unitarily isomorphic.

Since $\A$ and $\B$ are $\sigma$-unital C*-algebras, so are $\A_{\infty}$ and $\B_{\infty}$. Hence, every strong Morita equivalence $M$ for $E_{\infty}$ and $F_{\infty}$ must be a $\sigma$-TRO in the sense of \cite[p. 6]{EKK17}, so that a combination of \cite[Theorem 5.2]{EKK17} and \cite[Proposition 3.1]{EKK17} shows that $E_{\infty}\otimes \mathbb{K}$ and $F_{\infty}\otimes \mathbb{K}$ are unitarily equivalent if and only if $E_{\infty}$ and $F_{\infty}$ are strongly Morita equivalent.
\end{proof}

It is important to say a few words about the assumptions in Theorem \ref{t:down-and-up} and what they mean for C*-correspondences of adjacency matrices with entries in $\mathbb{N}$. Suppose $A$ is a square adjacency matrix with entries in $\mathbb{N}$, and indexed by a set $V$. First note that $V$ is countable if and only if $c_0(V)$ is $\sigma$-unital if and only if $C^*(G_A)_0$ is $\sigma$-unital. Moreover, by \cite[Proposition 8.8]{Rae05} the following holds

\begin{enumerate}

\item $A$ has finitely supported rows if and only if $G_A$ has finite out-degrees, if and only if the left action of $X(A)$ has image in $\K(X(A))$.

\item $A$ has no zero rows if and only if $G_A$ has no sinks, if and only if $X(A)$ has an injective left action.

\item $A$ has no zero columns if and only if $G_A$ has no sources, if and only if $X(A)$ is full.

\end{enumerate}

So we see that in order to apply Theorem \ref{t:down-and-up} to $X(A)$, we must verify that $A$ is over a countable set $V$, has finitely supported rows and is essential. In this case, Theorem \ref{t:down-and-up} shows that the existence of an equivariant stable isomorphism between graph C*-algebras $C^*(G_A)$ and $C^*(G_B)$ coincides with the existence of a strong Morita equivalence of the Pimsner dilations $X(A)_{\infty}$ and $X(B)_{\infty}$.

\section{Compatible shift equivalence} \label{s:cse}

In this section we introduce and study \emph{compatible shift equivalence}, which is formulated in terms of adjacency matrices and path isomorphisms. We show that it is indeed an equivalence relation, and that strong shift equivalence implies compatible shift equivalence. In what follows, for a matrix $A$ we write $\id_A$ to mean $\id_{E_A}$.

\begin{Def} \label{d:cse}
Let $A$ and $B$ be matrices indexed by $V$ and $W$ respectively, with entries in $\mathbb{N}$. Suppose there are a \emph{lag} $m\in \mathbb{N} \setminus \{0\}$ and matrices $R$ over $V\times W$ and $S$ over $W \times V$ with entries in $\mathbb{N}$ together path isomorphisms
$$
\phi_R : E_A \times E_R \rightarrow E_R \times E_B, \ \  \phi_S : E_B \times E_S \rightarrow E_S \times E_A,
$$
$$
\psi_A : E_R \times E_S \rightarrow E_A^m, \ \ \psi_B : E_S \times E_R \rightarrow E_B^m.
$$
We say that $R$ and $S$ are \emph{compatible} if
\begin{equation} \label{e:compat}
\phi_R^{(m)} = (\id_R \times \psi_B)(\psi_A^{-1} \times \id_R), \ \ \phi_S^{(m)} = (\id_S \times \psi_A)(\psi_B^{-1} \times \id_S),
\end{equation}
where
$$
\phi_R^{(m)} := (\phi_R \times \id_{B^{m-1}}) (\id_A \times \phi_R \times \id_{B^{m-2}})  \cdots  ( \id_{A^{m-1}} \times \phi_R)
$$
$$
\phi_S^{(m)} := (\phi_S \times \id_{A^{m-1}}) (\id_B \times \phi_S \times \id_{A^{m-2}}) \cdots ( \id_{B^{m-1}} \times \phi_S).
$$
Finally, we say that $A$ and $B$ are \emph{compatibly shift equivalent} if they are shift equivalent via a compatible pair of matrices $R$ and $S$.
\end{Def}

The following Lemma shows that certain relations between the maps $\phi_R, \phi_S,\psi_A,\psi_B$ are automatic when $R$ and $S$ are compatible.

\begin{Lemma} \label{l:ase}
Suppose $A$ and $B$ are essential matrices over $V$ and $W$ respectively, with entries in $\mathbb{N}$. Suppose $A$ and $B$ are compatibly shift equivalent with lag $m$, matrices $R$ and $S$, and path isomorphisms $\phi_R, \phi_S$, $\psi_A,\psi_B$. Then
\begin{gather*}
(\psi_A \times \id_{A}) (\id_R \times \phi_S) = (\id_A \times \psi_A) (\phi_R^{-1} \times \id_S) \\
(\psi_B \times \id_{B}) (\id_S \times \phi_R) = (\id_B \times \psi_B) (\phi_S^{-1} \times \id_R).
\end{gather*}
\end{Lemma}

\begin{proof}
First note that by compatible shift equivalence we have that
$$
(\psi_A \times \id_R \times \id_S) = (\phi_R^{(m)} \times \id_S)^{-1} (\id_R \times \psi_B \times \id_S) = 
$$
$$
(\phi_R^{(m)} \times \id_S)^{-1}( \id_R \times \phi_S^{(m)})^{-1} (\id_R \times \id_S \times \psi_A).
$$
However, since 
$$
(\psi_A \times \id_R \times \id_S)(\id_R \times \id_S \times \psi_A^{-1}) = (\id_{A^m} \times \psi_A^{-1})(\psi_A \times \id_{A^m}),
$$ 
we actually get that
\begin{equation} \label{eq:commute-phi-psi}
(\psi_A \times \id_{A^m}) = (\id_{A^m} \times \psi_A) (\phi_R^{(m)} \times \id_S)^{-1}( \id_R \times \phi_S^{(m)})^{-1}.
\end{equation}

Now let $(r,s,a) \in E_R \times E_S \times E_A$. Denote $r_0 = r, s_0 = s$ and $a_0 = a$. Since $A$ is essential, we can find $a_1,\cdots ,a_m, a_{m+1}, ..., a_{2m-1} \in E_A$ so that $a_0 a_1 \cdots a_m a_{m+1} \cdot ... \cdot a_{2m-1} \in E_A^{m+1}$. We may then define inductively $r_k \in E_R, s_k \in E_S$ and $a_k' \in E_A$ for $0\leq k \leq 2m-1$ so that
$$
(\phi_R^{-1} \times \id_S)(\id_R \times \phi_S^{-1})(r_k,s_k,a_k) = (a_k', r_{k+1},s_{k+1}).
$$
From equation \eqref{eq:commute-phi-psi} we get that $\psi_A(r_0, s_0) = a_0'\cdots a_{m-1}'$, that $\psi_A(r_1, s_1) = a_1'\cdots a_m'$ and that $a_0\cdots a_{m-1} = \psi_A(r_m,s_m) = a_m' \cdots a_{2m-1}'$. In particular, we see that $a= a_0 = a_m'$.

To prove the first equality in the statement, we compute
$$
(\id_A \times \psi_A) (\phi_R^{-1} \times \id_S) (\id_R \times \phi_S^{-1}) (r,s,a) = (\id_A \times \psi)(a_0',r_1,s_1) =
$$
$$
a_0' a_1' \cdots a_m' = \psi_A(r,s)a_m' = \psi_A(r,s)a = (\psi_A \times \id)(r,s,a).
$$
This shows that $(\id_A \times \psi_A) (\phi_R^{-1} \times \id_S) (\id_R \times \phi_S^{-1}) = (\psi_A \times \id)$, which is then equivalent to the first equality $(\psi_A \times \id_{A}) (\id_R \circ \phi_S) = (\id_A \times \psi_A) (\phi_R^{-1} \times \id_S)$. A symmetric argument works to show the second equality as well. 
\end{proof}

\begin{Rmk}
Notice from the proof that we also get that $\psi_A$ and $\psi_B$ are uniquely determined by $\phi_R$ and $\phi_S$. Indeed, from equation \eqref{eq:commute-phi-psi} we get that
$$
(\psi_A \times \psi_A^{-1}) = (\phi_R^{(m)} \times \id_S)^{-1}( \id_R \times \phi_S^{(m)})^{-1},
$$
so that compressing to the first part yields back $\psi_A$. A similar argument then works for $\psi_B$ as well.
\end{Rmk}

\begin{Prop} \label{p:cse-transitive}
Compatible shift equivalence is an equivalence relation on the collection of essential matrices with entries in $\mathbb{N}$.
\end{Prop}

\begin{proof}
That compatible shift equivalence is reflexive and symmetric is clear. Thus, we need to show transitivity.

Let $A, B$ and $C$ be essential matrices over $\mathbb{N}$. Suppose that $A$ and $B$ are compatibly shift equivalent with lag $m$, matrices $R,S$ and path isomorphisms $\psi_A,\psi_B,\phi_R,\phi_S$, while $B$ and $C$ are compatibly shift equivalent with lag $m'$, matrices $R',S'$ and path isomorphisms $\psi_B', \psi_C', \phi_{R'}, \phi_{S'}$. We claim that $A$ and $C$ are compatibly shift equivalent with lag $m+m'$, matrices $RR', S'S$ and path isomorphisms $\psi_A', \psi_C, \phi_{RR'},\phi_{S'S}$ given by
\begin{gather*}
\psi_A' := (\id_{A^{m'}} \times \psi_A)((\phi_R^{(m')})^{-1} \times \id_S)(\id_R \times \psi_B' \times \id_S), \\
\psi_C:= (\id_{C^m} \times \psi_C')((\phi_{S'}^{(m)})^{-1} \times \id_{R'})(\id_{S'} \times \psi_B \times \id_{R'}), \\
\phi_{RR'}:= (\id_R \times \phi_{R'})(\phi_R \times \id_{R'}), \\
\phi_{S'S}:= (\id_S \times \phi_{S'})(\phi_S \times \id_{S'}).
\end{gather*}

It is easy to see that $RR'S'S = A^{m+m'}$, $S'SRR' = C^{m+m'}$, that $RR' C = ARR'$ and $CS'S =S'S A$. Moreover, it is clear that $\psi_A', \psi_C, \phi_{RR'}, \phi_{S'S}$ are path isomorphisms. Thus, in order to show that the above data comprises a compatible shift equivalence, we need only show
\begin{gather*}
\phi_{RR'}^{(m+m')} (\psi_A' \times \id_{RR'}) = \id_{RR'} \times \psi_C \\
\phi_{S'S}^{(m+m')} (\psi_C \times \id_{S'S}) = \id_{S'S} \times \psi_A'.
\end{gather*}
We will show the first of these equalities, and the second will follow from a symmetric argument.

First, let $r_1r_1's'sr_2r_2' \in E_{RR'} \times E_{S'S} \times E_{RR'}$, and denote 
\begin{equation} \label{eq:1}
\mu_Ar_3:= (\phi_R^{(m')})^{-1}(r_1\psi_B'(r_1's')) \ \ \text{and} \ \ r_3' \mu_C:= \phi_{R'}^{(m)}(\psi_B(sr_2)r_2').
\end{equation}
Then we get that
$$
\psi_A'(r_1r_1's's) = (\id_{A^m} \times \psi_A)((\phi_R^{(m)})^{-1} \times \id_S)(r_1 \psi_B'(r's')s) =
$$
$$
(\id_{A^m} \times \psi_A)(\mu_Ar_3 s) = \mu_A\psi_A(r_3s),
$$
and from Lemma \ref{l:ase} we also get that
$$
\psi_C(s'sr_2r_2') = (\id_{C^m} \times \psi_C')( (\phi_{S'}^{(m)})^{-1} \times \id_{R'})(s' \psi_B(sr_2)r_2') =
$$
$$
(\psi_C' \times \id_{C^m})(\id_{S'} \times \phi_{R'}^{(m)})(s' \psi_B(sr_2)r_2') = \psi_C'(s'r_3')\mu_C.
$$
Thus, together we obtained
\begin{equation} \label{eq:2}
\psi_A'(r_1r_1's's) = \mu_A\psi_A(r_3s) \ \ \text{and} \ \ \psi_C(s'sr_2r_2') = \psi_C'(s'r_3')\mu_C.
\end{equation}
Next, from compatibility we also get
\begin{equation} \label{eq:3}
\phi_R^{(m)}(\psi_A(r_3s)r_2)) = r_3\psi_B(sr_2) \ \ \text{and} \ \ \phi_{R'}^{(m)}(\psi_B'(r_1's')r_3') = r_1' \psi_C'(s'r_3').
\end{equation}
Combining equations \eqref{eq:1}, \eqref{eq:2}, \eqref{eq:3}, we compute

\begin{eqnarray*}
\lefteqn{\phi_{RR'}^{(m+m')}(\psi_A' \times \id_{RR'})(r_1r_1's'sr_2r_2')}\\
&=& \phi_{RR'}^{(m+m')}(\mu_A\psi_A(r_3s)r_2r_2')\\
&=& (\id_R \times \phi_{R'}^{(m+m')})(\phi_R^{(m+m')} \times \id_{R'})(\mu_A \psi_A(r_3s)r_2r_2') \\
&=& (\id_R \times \phi_{R'}^{(m+m')})(\phi_R^{(m')} \times \id_{B^m} \times \id_{R'})(\mu_A r_3 \psi_B(sr_2) r_2')\\
&=& (\id_R \times \phi_{R'}^{(m+m')})(r_1\psi_B'(r_1's')\psi_B(sr_2)r_2') \\
&=& (\id_R \times \phi_{R'}^{(m')} \times \id_{C^m})(r_1\psi_B'(r_1's')r_3'\mu_C)\\
&=& r_1r_1' \psi_C'(s'r_3')\mu_C\\
&=& r_1r_1'\psi_C(s'sr_2r_2') \\
&=& (\id_{RR'} \times \psi_C)(r_1r_1's'sr_2r_2').
\end{eqnarray*}

Thus, we have shown $\phi_{RR'}^{(m+m')}(\psi_A' \times \id_{RR'}) = \id_{RR'} \times \psi_C$ as desired.

\end{proof}

\begin{Cor} \label{c:sse_implies_cse}
Let $A$ and $B$ be essential matrices over $V$ and $W$ respectively, with entries in $\mathbb{N}$. If $A$ and $B$ are strong shift equivalent, then they are compatibly shift equivalent.
\end{Cor}

\begin{proof}
By Proposition \ref{p:cse-transitive} we know that compatible shift equivalence is an equivalence relation. Hence, it will suffice to show that if $A$ and $B$ are elementary shift related via $R$ and $S$, then $R$ and $S$ are compatible.

Since $A$ and $B$ are elementary shift related via $R$ and $S$, we have that $A = RS$ and $B= SR$. So choose some path isomorphisms $\psi_A : E_R \times E_S \rightarrow E_A$ and $\psi_B : E_S \times E_S \rightarrow E_B$ and define
$$
\phi_R : E_A \times E_R \rightarrow E_R \times E_A \ \ \text{and} \ \ \phi_S : E_B \times E_S \rightarrow E_S \times E_A
$$
by setting $\phi_R := (\id_R \times \psi_B)(\psi_A^{-1} \times \id_R)$ and $\phi_S := (\id_S \times \psi_A)(\psi_B^{-1} \times \id_S)$. Since the lag is $m=1$, compatibility follows by definition of $\phi_R$ and $\phi_S$.
\end{proof}

\begin{Rmk}

When $A$ and $B$ are essential matrices with entries in $\mathbb{N}$ over $V$ and $W$ respectively, it can be shown directly that compatible shift equivalence implies conjugacy of $(X_{E_A}, \sigma_{E_A})$ and $(X_{E_B},\sigma_{E_B})$ with a formula for the homeomorphism $h : X_{E_A} \rightarrow X_{E_B}$ which implements the conjugacy. Thus, by Williams' theorem, it follows that $A$ and $B$ are strong shift equivalent. We skip the argument here, because it will follow from Theorem \ref{t:main} that the converse of Corollary \ref{c:sse_implies_cse} holds.
\end{Rmk}

\section{Representable shift equivalence} \label{s:rse}

Our goal in this section is to determine when a shift equivalence between two matrices can be represented as operators on Hilbert space. This leads to the notion of representable shift equivalence. We show that compatible shift equivalence implies representable shift equivalence, and that a representation of shift equivalence can be chosen so that the graph C*-algebras act \emph{faithfully} on the Hilbert space.

Let $A$ and $B$ be matrices over $\mathbb{N}$ indexed by $V$ and $W$ respectively, and suppose that $R$ and $S$ are matrices over $\mathbb{N}$ indexed by $V \times W$ and $W\times V$ respectively. To ease some of our notation, we define two matrices indexed by $V \sqcup W$,
\begin{equation*}
C = \begin{bmatrix} A & 0 \\ 0 & B \end{bmatrix} \ \ \text{and}  \ \ D = \begin{bmatrix} 0 & R \\ S & 0 \end{bmatrix},
\end{equation*}
so that $A$ and $B$ are shift equivalent with lag $m$ via $R$ and $S$ if and only if 
$$
CD = DC \ \ \text{and} \ \ D^2 =C^m.
$$
From shift equivalence, there exist path isomorphisms
$$
\phi_R : E_A \times E_R \rightarrow E_R \times E_B, \ \  \phi_S : E_B \times E_S \rightarrow E_S \times E_A,
$$
$$
\psi_A : E_R \times E_S \rightarrow E_A^m, \ \ \psi_B : E_S \times E_R \rightarrow E_B^m.
$$
We may define path isomorphisms $\phi : E_C \times E_D \rightarrow E_D \times E_C$ and $\psi: E_D^2 \rightarrow E_C^m$ given by
\begin{equation*}
\psi = \begin{bmatrix} \psi_A & 0 \\ 0 & \psi_B \end{bmatrix} \ \ \text{and}  \ \ \phi = \begin{bmatrix} 0 & \phi_R \\ \phi_S & 0 \end{bmatrix}.
\end{equation*}
Conversely, all path isomorphisms $\phi : E_C \times E_D \rightarrow E_D \times E_C$ and $\psi: E_D^2 \rightarrow E_C^m$ must be of the above form for some path isomorphisms $\phi_R, \phi_S, \psi_A, \psi_B$ as above. Hence, we see that $A$ and $B$ are compatibly shift equivalent with lag $m$ via $R$ and $S$ if and only if there exist path isomorphisms $\phi : E_C \times E_D \rightarrow E_D \times E_C$ and $\psi: E_D^2 \rightarrow E_C^m$ such that
\begin{equation} \label{e:compatible}
\phi^{(m)} = (\id_D \times \psi) (\psi^{-1} \times \id_D),
\end{equation}
where $\phi^{(m)} : E_C^m \times E_D \rightarrow E_D \times E_C^m$ is the path isomorphism given by
$$
\phi^{(m)} := (\phi \times \id_{C^{m-1}}) (\id_{C} \times \phi \times \id_{C^{m-2}}) \cdots ( \id_{C^{m-1}} \times \phi).
$$

The following is the natural way to represent shift equivalence as bounded operators on Hilbert space, via some choice of path isomorphisms as above.

\begin{Def} \label{d:rse}
Let $A$ and $B$ be square matrices indexed by $V$ and $W$ respectively, with entries in $\mathbb{N}$. Let $m\in \mathbb{N} \setminus \{0\}$ and suppose there are matrices $R$ over $V\times W$ and $S$ over $W \times V$ with entries in $\mathbb{N}$ for which there exist path isomorphisms
$$
\phi : E_C \times E_D \rightarrow E_D \times E_C \ \ \ \text{and} \ \ \ \psi: E_D^2 \rightarrow E_C^m.
$$
We say that $R$ and $S$ are \emph{representable} via $\phi$ and $\psi$ if there are Cuntz--Krieger families $(P_v,S_c)$ for $G_C$ and $(P_v,T_d)$ for $G_D$ on the same Hilbert space $\mathcal{H}$ with $P_v \neq 0$ for all $v \in V \sqcup W$ such that
\begin{equation} \label{eq:psi}
T_{d_1d_2} = S_{c_1 \cdots c_m} \ \ \text{when} \ \ \psi(d_1d_2) = c_1 \cdots c_m \ \ \text{for} \ \ d_1d_2\in E_D^2,
\end{equation}
\begin{equation} \label{eq:phi}
S_cT_d = T_{d'}S_{c'} \ \ \text{when} \ \ \phi(cd) = d'c' \ \ \text{for} \ \ cd\in E_C\times E_D.
\end{equation}
We say that $A$ and $B$ are \emph{representable shift equivalent} if there are $R$ and $S$ and $\phi$ and $\psi$ as above so that $R$ and $S$ are representable via $\phi$ and $\psi$.
\end{Def}

\begin{Rmk} \label{r:same-vert-ops}
When $A$ and $B$ both have no zero columns we may define $R$ and $S$ to be representable via $\phi$ and $\psi$ by specifying apriori separate Cuntz-Krieger families $(S_v,S_c)$ for $G_C$ and $(T_w,T_d)$ for $G_D$. It then follows from equation \eqref{eq:psi} that $S_v = T_v$ for every $v\in V\sqcup W$. Indeed, since no vertex is a source, for every $v\in V$ there exists $c_1\cdots c_m \in E_C^m$ with $r(c_m) = v$, so we may take $\psi^{-1}(c_1 \cdots c_m) = d_1d_2$ with $d_1d_2\in E_D^2$ and $r(d_2) = v$. So that
$$
S_v = S_{c_1\cdots c_m}^*S_{c_1\cdots c_m} = T_{d_1d_2}^*T_{d_1d_2} = T_v.
$$
\end{Rmk}

\begin{Prop} \label{p:representable}
Let $A$ and $B$ be essential matrices indexed by $V$ and $W$ respectively, with entries in $\mathbb{N}$ and finitely supported rows. Suppose there are a lag $m \in \mathbb{N} \setminus \{0\}$ and matrices $R$ over $V\times W$ and $S$ over $W \times V$ with entries in $\mathbb{N}$, and path isomorphisms 
$$
\phi_R : E_A \times E_R \rightarrow E_R \times E_B, \ \  \phi_S : E_B \times E_S \rightarrow E_S \times E_A,
$$
$$
\psi_A : E_R \times E_S \rightarrow E_A^m, \ \ \psi_B : E_S \times E_R \rightarrow E_B^m
$$
such that
\begin{equation} \label{rb=ar'}
(\psi_A \times \id_A) (\id_R \times \phi_S) = (\id_A \times \psi_A)  (\phi_R^{-1} \times \id_S) \ \ \text{and}
\end{equation}
\begin{equation} \label{sr=b_1...b_m}
\phi_S^{(m)} = (\id_S \times \psi_A)(\psi_B^{-1} \times \id_S).
\end{equation}
Then $R$ and $S$ are representable via $\phi$ and $\psi$.
\end{Prop}

\begin{proof}
Let $X_{E_D}^+$ be the one-sided subshift for $G_D = (V\sqcup W, E_D)$, and let $\{e_x\}_{x\in X_{E_D}^+}$ be an orthonormal basis for $\ell^2(X_{E_D}^+)$. We denote by $VX_{E_D}^+$ and $WX_{E_D}^+$ the clopen subsets of those $x\in X_{E_D}^+$ such that $s(x) \in V$ and $s(x) \in W$ respectively. We then have the homeomorphisms $\psi_A^{\infty} : VX_{E_D}^+ \rightarrow X_{E_A}^+$ and $\psi_B^{\infty} : WX_{E_D}^+ \rightarrow X_{E_B}^+$ given by 
$$
\psi_A^{\infty}(r_0s_0r_1s_1 \cdots ) = \psi_A(r_0s_0)\psi_A(r_1s_1) \cdots ,
$$
$$
\psi_B^{\infty}(s_0r_0s_1r_1\cdots ) = \psi_B(s_0r_0)\psi_B(s_1r_1)\cdots 
$$
for $r_i \in E_R$ and $s_i \in E_S$.

We define $P_v$ for $v\in V \sqcup W$ by
$$
P_v(e_x) = \begin{cases}
e_x & \text{if }v=s(x),\\
0 & \text{otherwise.}
\end{cases}
$$

For $a \in E_A$ and $x\in X_{E_D}^+$, and define $S_a$ via
$$
S_a(e_x) = \begin{cases}
e_y & \text{if } r(a)=s(x), \\
0 & \text{otherwise},
\end{cases}
$$
where $y := (\psi_A^{\infty})^{-1}(a\psi_A^{\infty}(r_0s_0r_1s_1\cdots ))$ when we write $x= r_0s_0\cdots $ for elements $r_i \in E_R$ and $s_i\in E_S$ in case that $r(a) = s(x)$.

For $b\in E_B$, $x\in X_{E_D}^+$ and define $S_b$ via
$$
S_b(e_x) = \begin{cases}
e_y & \text{if } r(b)=s(x), \\
0 & \text{otherwise}
\end{cases}
$$
where $y := (\id_S \times \psi_A^{\infty})^{-1}(\phi_S(bs_0)\psi^{\infty}_A(r_0s_1r_1s_2 \cdots ))$ when we write $x= s_0r_0 \cdots $ for elements $s_i \in E_S$, $r_i\in E_R$ in the case that $r(b)=s(x)$.

Finally, we define for $d\in E_D$ the operator $T_d$ via
$$
T_d(e_x) = \begin{cases}
e_{dx} & \text{if } r(d) = s(x), \\
0 & \text{otherwise}.
\end{cases}
$$
It is straightforward to verify that $(P_v,T_d)$ is a Cuntz-Krieger family for $G_D$. So we first verify that $(P_v,S_c)$ is a Cuntz-Krieger family for $G_C$. Since concatenation for $S_a$ is done simply through the homeomorphism $\psi_A^{\infty}$, it is easy to show that for $a\in E_A$ or $v\in V$, we have $S_a^*S_a = P_{r(a)}$, and $\sum_{s(e)=v} S_eS_e^* = P_v$. Now, to show the same for $b\in E_B$, for $y=s_0'r_0's_1'r_1' \cdots$ when $s_i' \in E_S$ and $r_i' \in E_R$ we write $(\id_B \times \id_S \times \psi_A^{\infty})^{-1}(\phi_S^{-1}\times \id)(s_0'\psi_A^{\infty}(r_0's_1'r_1' \cdots)) = b's_0r_0 \cdots$ for $b'\in E_B$, $s_i \in E_S$, $r_i \in E_R$, so that
$$
S_b^*(e_y) = \begin{cases}
e_{s_0r_0 \cdots} & \text{if } r(b)=s(x), b=b',\\
0 & \text{otherwise},
\end{cases}
$$
From this formula it follows that $S_b^*S_b = P_{r(b)}$, and that for $w\in W$ we have $\sum_{s(e) = w} S_eS_e^* = P_w$. Thus, we see that $(P_v,S_c)$ is a Cuntz-Krieger family for $G_C$. Since clearly $P_v \neq 0$ for every $v\in V \sqcup W$, we are left with verifying equations \eqref{eq:psi} and \eqref{eq:phi}.

It is clear from the definition of $S_a$ for $a\in E_A$ that for $r\in E_R$ and $s\in E_S$ we have $T_{rs} = S_{a_1 \cdots a_m}$ when $\psi_A(rs) =a_1 \cdots a_m$. We next show that $S_bT_s = T_{s'}S_a$ when $\phi_S(bs) = s'a$ for $b\in E_B$ and $s\in E_S$. Indeed, let $x\in X_{E_D}^+$ with $s(x) \in V$, and write $x = r_0s_0 \cdots$ so that $S_bT_s(e_x) = e_z$ and $T_{s'}S_a(e_x) = e_{z'}$ where
$$
z = (\id_S \times \psi_A^{\infty})^{-1}\phi_S(bs)\psi_A^{\infty}(r_0s_0 \cdots) 
$$
and
$$
z' = s'(\psi_A^{\infty})^{-1}(a\psi_A^{\infty}(r_0s_0 \cdots)).
$$
Since $\phi(bs) = s'a$, it follows that $z=z'$, so that $S_bT_s = T_{s'}S_a$.

Next, we show that $T_{sr} = S_{b_1 \cdots b_m}$ when $\psi_B(sr) = b_1\cdots b_m$. Indeed, let $x\in X_{E_D}^+$ with $s(x) \in W$, and write $x=s_0r_0 \cdots$. Then we have that $S_{b_1\cdots b_m}(e_x) = e_z$ where
$$
z = (\id_S \times \psi_A^{\infty})^{-1}(\phi_S^{(m)}(b_1 \cdots b_ms_0)\psi_A^{\infty}(r_0s_1 \cdots ))
$$
From equation \eqref{sr=b_1...b_m} it follows that
$$
z = (\psi_B)^{-1}(b_1\cdots b_m)s_0r_0s_1 \cdots
$$
Thus, we see that if $\psi_B(sr) = b_1 \cdots b_m$, then $T_{sr} = S_{b_1 \cdots b_m}$.

Finally, we show that when $\phi_R(ar)=r'b$ we have $S_aT_r = T_{r'}S_b$. Indeed, let $x\in X_{E_D}^+$ with $s(x) \in W$, and write $x=s_0r_0 \cdots$ so that $S_aT_r(e_x) = e_z$ and $T_{r'}S_b(e_x) = e_{z'}$ where
$$
z = (\psi_A^{\infty})^{-1}(a\psi_A^{\infty}(rs_0r_0\cdots ))
$$
and
$$
z' = r'(\id_S \times \psi_A^{\infty})^{-1}(\phi_S(bs_0)\psi_A^{\infty}(r_0s_1r_1\cdots ))
$$
But by equation \eqref{rb=ar'} and the fact that $\phi_R(ar) = r'b$ we get that
$$
z' = (\psi_A^{\infty})^{-1}(\id_A \times \psi_A)(\phi_R^{-1} \times \id_S)(r'bs_0)\psi_A^{\infty}(r_0s_1r_1\cdots ) =
$$
$$
(\psi_A^{\infty})^{-1}(\id_A \times \psi_A)(ars_0 \cdot \psi_A^{\infty}(r_0s_1r_1\cdots )) = (\psi_A^{\infty})^{-1}(a\psi_A^{\infty}(rs_0r_0\cdots )) = z.
$$
Hence, we get that $z'=z$, so that $S_aT_r = T_{r'}S_b$. Thus, we have shown that that $R$ and $S$ are representable via $\phi$ and $\psi$.
\end{proof}

\begin{Rmk}
We note that for two essential matrices $A$ and $B$ with entries in $\mathbb{N}$ to be representable shift equivalent, we need only know the validity of the two asymmetric equations \eqref{rb=ar'} and \eqref{sr=b_1...b_m}, as opposed to the symmetric equations in \eqref{e:compat}.
\end{Rmk}

Using Szyma\'{n}ski's uniqueness theorem \cite[Theorem 1.2]{Szy02}, we can upgrade a representation of a shift equivalence to be injective in the following sense.

\begin{Cor} \label{c:cse-implies-rse}
Let $A$ and $B$ be essential matrices indexed by $V$ and $W$ respectively, with entries in $\mathbb{N}$ and finitely supported rows. Suppose that $A$ and $B$ are compatibly shift equivalent via $R$ and $S$ and path isomorphisms $\phi$ and $\psi$. Then $R$ and $S$ are representable via $\phi$ and $\psi$. In fact, there are Cuntz-Krieger families $(P_v,S_c)$ for $G_C$ and $(P_v,T_d)$ for $G_D$ satisfying equations \eqref{eq:psi} and \eqref{eq:phi} so that both of the canonical surjections $q_C: C^*(G_C) \rightarrow C^*(P_v,S_c)$ and $q_D : C^*(G_D) \rightarrow C^*(P_v,T_d)$ are injective.
\end{Cor}

\begin{proof}
From Lemma \ref{l:ase} equation \eqref{rb=ar'} holds. Since equation \eqref{sr=b_1...b_m} holds by definition, by Proposition \ref{p:representable} there are Cuntz-Krieger families $S:=(P_v,S_c)$ for $G_C$ and $T:=(P_v,T_d)$ for $G_D$ satisfying equations \eqref{eq:psi} and \eqref{eq:phi}.

Let $z\in \mathbb{T}$ be some unimodular scalar. Then we may define two operator families $S^{(z)}:= (P_v,z\cdot S_c)$ and $T^{(z)}:=(P_v,T_d^{(z)})$ where 
$$
T_d^{(z)}:= \begin{cases}
T_d & \text{if } d\in E_R \\
z^m \cdot T_d & \text{if } d \in E_S
\end{cases}
$$
Then clearly $S^{(z)}$ and $T^{(z)}$ are Cuntz-Krieger families satisfying equations \eqref{eq:psi} and \eqref{eq:phi}. 

Let $(z_n)_{n\in \mathbb{N}}$ be a countable dense subset of $\mathbb{T}$. We take $S':= \oplus_{n=1}^{\infty}S^{(z_n)}$ and $T':= \oplus_{n=1}^{\infty}T^{(z_n)}$, which are still Cuntz-Krieger families satisfying equations \eqref{eq:psi} and \eqref{eq:phi}. By Szyma\'{n}ski's uniqueness theorem \cite[Theorem 1.2]{Szy02}, it suffices to show that for every cycle $c_1\cdots c_{\ell}$ without exits in $G_C$, and every cycle $d_1\cdots d_{2t}$ without exits in $G_D$ (which is necessarily of even length since $G_D$ is bipartite), the spectrum of the operators $S_{c_1\cdots c_{\ell}}$ and $T_{d_1\cdots d_{2t}}$ contains the entire unit circle. 

Since $S_{c_1\cdots c_{\ell}}$ and $T_{d_1\cdots d_{2t}}$ are unitaries on the ranges of $P_{s(c_1)}$ and $P_{s(d_1)}$ respectively, each of their spectra must contain some element in the unit circle. But since for every $n \in \mathbb{N}$ we have that $z_n^{\ell} \cdot S_{c_1\cdots c_{\ell}}$ is a direct summand of $S'_{c_1\cdots c_{\ell}}$, and $z_n^{m \cdot t} \cdot T_{d_1 \cdots d_{2t}}$ is a direct summand of $T'_{d_1\cdots d_{2t}}$, we see that the spectra of $S'_{c_1\cdots c_{\ell}}$ and $T'_{d_1\cdots d_{2t}}$ both contain a dense subset of $\mathbb{T}$, and hence $\mathbb{T}$ itself. Thus, by Szyma\'{n}ski's uniqueness theorem we get that the canonical surjections $C^*(G_C) \rightarrow C^*(P'_v,S'_c)$ and $C^*(G_D) \rightarrow C^*(P'_v,T'_d)$ are injective.
\end{proof}

\begin{Rmk}
Suppose we have two essential matrices $A$ and $B$ with entries in $\mathbb{N}$ indexed by $V$ and $W$ respectively, and that $R,S$ are matrices that comprise a representable shift equivalence of lag $m$ via path isomorphisms $\phi$ and $\psi$. It can be shown directly that $A$ and $B$ are compatible shift equivalent with lag $m$ via $R$ and $S$, together with the $1-1$ and $2-2$ corners of $\psi$ and the $1-2$ and $2-1$ corners of $\phi$. We skip the proof since representable shift equivalence implies compatible shift equivalence by Theorem \ref{t:main}.
\end{Rmk}

\section{Strong Morita shift equivalence} \label{s:smse}

In this section we introduce and study strong Morita shift equivalence. This equivalence relation is expressed in terms of a specific strong Morita equivalence between Pimsner dilations, and is implied by representable shift equivalence. Strong Morita shift equivalence turns out to imply the existence of a stable equivariant isomorphism of graph C*-algebras that also preserves the diagonal subalgebras.

Suppose $A$ and $B$ are essential matrices over $\mathbb{N}$ indexed by $V$ and $W$ respectively, and have finitely supported rows. Suppose there are matrices $R$ and $S$ with entries in $\mathbb{N}$, and let $C$ and $D$ be as described in the previous section. Suppose further that we (only) have a path isomorphism $\psi : E_D^2 \rightarrow E_C^m$, which is then the direct sum of path isomorphisms 
$$
\psi_A : E_R \times E_S \rightarrow E_A^m, \ \ \text{and} \ \ \psi_B : E_S \times E_R \rightarrow E_B^m.
$$

Now let $(S_v, S_c)$ be a CK family that generates the graph C*-algebra $C^*(G_C) = C^*(G_A) \oplus C^*(G_B)$ of $G_C$, and let $(T_v,T_d)$ be a CK family that generates the graph C*-algebra $C^*(G_D)$. We denote
$$
\A_n^C:= \overline{\Span}\{ \ S_{\lambda}S_{\lambda'}^* \ | \ \lambda,\lambda' \in E_C^n, \ r(\lambda) \in V \ \},
$$
$$
\B_n^C:= \overline{\Span}\{ \ S_{\lambda}S_{\lambda'}^* \ | \ \lambda,\lambda' \in E_C^n, \ r(\lambda) \in W \ \},
$$
$$
\A_n^D := \overline{\Span} \{ \ T_{\mu}T_{\mu'}^* \ | \ \mu, \mu' \in E_D^n , \ r(\mu) \in V \ \},
$$
$$
\B_n^D := \overline{\Span} \{ \ T_{\mu}T_{\mu'}^* \ | \ \mu, \mu' \in E_D^n , \ r(\mu) \in W \ \}.
$$
with direct limits $\A_{\infty}^C$, $\B_{\infty}^C$, $\A_{\infty}^D$ and $\B_{\infty}^D$. Then it follows that $\A_{\infty}^C \oplus \B_{\infty}^C$ is canonically isomorphic to the fixed point algebra $C^*(G_C)_0$ of $C^*(G_C)$ with its canonical gauge action, and that $\A_{\infty}^D \oplus \B_{\infty}^D$ is canonically isomorphic to the fixed point algebra $C^*(G_D)_0$ of $C^*(G_D)$ with its canonical gauge action.

Thinking of $X(A)$ and $X(B)$ as block diagonal C*-subcorrespondences of $X(C)$ in the natural way, we get that the Pimsner dilations $X(A)_{\infty}$ and $X(B)_{\infty}$ are naturally identified with the direct limits of C*-correspondences
$$
X(A)_n:= \overline{\Span}\{ \ S_{\lambda}S_{\lambda'}^* \ | \ \lambda \in E_C^{n+1} ,\lambda' \in E_C^n, \ s(\lambda) \in V \ \},
$$
which is a C*-correspondence from $\A_n^C$ to $\A_{n+1}^C$, and
$$
X(B)_n:= \overline{\Span}\{ \ S_{\lambda}S_{\lambda'}^* \ | \ \lambda \in E_C^{n+1} ,\lambda' \in E_C^n, \ s(\lambda) \in W \ \}
$$
which is a C*-correspondence from $\B_n^C$ to $\B_{n+1}^C$.

Similarly, thinking of $X(R)$ and $X(S)$ as the off-diagonal C*-subcorres\-pondences of $X(D)$ in the natural way, we get the C*-correspondences (which are only denoted as) $X(R)_{\infty}$ and $X(S)_{\infty}$ as the direct limits of C*-correspondences
$$
X(R)_n:= \overline{\Span}\{ \ T_{\mu}T_{\mu'}^* \ | \ \mu \in E_D^{n+1} ,\mu' \in E_D^n, \ s(\mu) \in V \ \}
$$
which is a C*-correspondence from $\B_n^D$ to $\A_{n+1}^D$, and
$$
X(S)_n:= \overline{\Span}\{ \ T_{\mu}T_{\mu'}^* \ | \ \mu \in E_D^{n+1} ,\mu' \in E_D^n, \ s(\mu) \in W \ \}
$$
which is a C*-correspondence from $\A_n^D$ to $\B_{n+1}^D$. Note however that $X(R)_{\infty}$ and $X(S)_{\infty}$ are not Pimsner dilations in the sense described in Section \ref{s:SE-CP}, because they are C*-correspondences over possibly different left and right coefficient C*-algebras. 

Recall that $X(D)_{\infty}$ is the C*-correspondence over $C^*(G_D)_0 = \A_{\infty}^D \oplus \B_{\infty}^D$ given by $X(D)_{\infty} = X(D) \cdot (\A_{\infty}^D \oplus \B_{\infty}^D)$, which is identified with the direct limit
$$
\overline{\Span}\{\ T_{\mu}T_{\mu'}^* \ | \ \mu \in E_D^{n+1} ,\mu' \in E_D^n, \ n \in \mathbb{N} \ \}.
$$ 
The $\B_{\infty}^D - \A_{\infty}^D$ correspondence $X(R)_{\infty}$ and the $\A_{\infty}^D - \B_{\infty}^D$ correspondence $X(S)_{\infty}$ coincide with the $1-2$ corner and $2-1$ corner (respectively) of the Pimsner dilation $X(D)_{\infty}$, and satisfy the equalities
$$
X(R)_{\infty} = X(R) \cdot \B^D_{\infty}, \ \ X(S)_{\infty} = X(S) \cdot \A^D_{\infty}.
$$

Now, using the map $\psi : E_D^2 \rightarrow E_C^m$, for each $k\in \mathbb{N}$ and $\mu_1,\cdots, \mu_k \in E_D^2$ such that $\mu_1\cdots \mu_k \in E_D^{2k}$ we may define $\psi_k : E_D^{2k} \rightarrow E_C^{mk}$ by setting $\psi_k(\mu_1 \cdots \mu_k) = \psi(\mu_1) \cdots \psi(\mu_k)$. This then gives rise to $*$-isomorphisms $\Psi^{A}_k :\A_{2k}^D \rightarrow \A_{mk}^C$ and $\Psi^{B}_k : \B_{2k}^D \rightarrow \B_{mk}^C$ by setting $\Psi^{A}_k(T_{\mu}T_{\mu_1'}^*) = S_{\psi_k(\mu)} S_{\psi_k(\mu')}^*$, and $\Psi^{B}_k(T_{\mu}T_{\mu'}^*) = S_{\psi_k(\mu)} S_{\psi_k(\mu')}^*$. Since these maps are compatible with direct limits, we obtain two $*$-isomor\-phisms $\Psi_{\infty}^A : \A_{\infty}^D \rightarrow \A_{\infty}^C$ and $\Psi_{\infty}^B : \B_{\infty}^D \rightarrow \B_{\infty}^C$ that we will use as identifications between the coefficient C*-algebras. For instance, this allows us to turn $X(R)_{\infty}$ into a $\B_{\infty}^C - \A_{\infty}^C$-bimodule, where the left and right actions are implemented via $\Psi_{\infty}^A$ and $\Psi_{\infty}^B$ respectively, and the inner product via $\Psi_{\infty}^B$.

\begin{Def} \label{d:smse}
Let $A$ and $B$ be essential matrices over $\mathbb{N}$ indexed by $V$ and $W$ respectively, with finitely supported rows. We say that $X(A)_{\infty}$ and $X(B)_{\infty}$ are \emph{strong Morita shift equivalent} if there are a lag $m\in \mathbb{N} \setminus \{0\}$ and matrices $R$ over $V\times W$ and $S$ over $W \times V$ over $\mathbb{N}$ together with path isomorphisms 
$$
\psi_A : E_R \times E_S \rightarrow E_A^m, \ \ \psi_B : E_S \times E_R \rightarrow E_B^m
$$ 
such that $X(R)_{\infty}$ is a strong Morita equivalence between $X(A)_{\infty}$ and $X(B)_{\infty}$, up to the identifications $\Psi_{\infty}^A$ and $\Psi_{\infty}^B$. More precisely, when $X(R)_{\infty}$ is considered as a $\B_{\infty}^C - \A_{\infty}^C$-bimodule via $\Psi_{\infty}^A$ and $\Psi_{\infty}^B$, there exists a unitary bimodule isomorphism $U: X(A)_{\infty} \otimes X(R)_{\infty} \rightarrow X(R)_{\infty} \otimes X(B)_{\infty}$ such that for every $a \in \A_{\infty}^C$, $b\in \B_{\infty}^C$ and $\xi \in X(A)_{\infty} \otimes X(R)_{\infty}$ we have
$$
U(a \cdot \xi \cdot b) = a \cdot U(\xi) \cdot b.
$$
\end{Def}

Henceforth, we will no longer belabor the point of distinguishing between $\A_{\infty}^C$ and $\A_{\infty}^D$ and between $\B_{\infty}^C$ and $\B_{\infty}^D$. However, we emphasize that this identification is important in the above definition, and depends on the choice of the maps $\psi_A$ and $\psi_B$. We have already seen that
$$
X(A)_{\infty} =X(A) \cdot \A_{\infty},  \ X(B)_{\infty} = X(B) \cdot \B_{\infty},
$$
$$
X(R)_{\infty} =X(R) \cdot \B_{\infty},  \ X(S)_{\infty} = X(S) \cdot \A_{\infty},
$$
so that by the above discussion and identifications using $\Psi_{\infty}^A$ and $\Psi_{\infty}^B$, as $\B_{\infty} - \A_{\infty}$ correspondences we may canonically identify
$$
X(A)_{\infty} \otimes_{\A_{\infty}} X(R)_{\infty} \cong X(A) \otimes_{\A} X(R) \cdot \B_{\infty},
$$
$$
X(R)_{\infty} \otimes_{\B_{\infty}} X(B)_{\infty} \cong X(R) \otimes_{\B} X(B) \cdot \B_{\infty}.
$$

\begin{Prop} \label{p:rse-implies-sme}
Let $A$ and $B$ be essential matrices over $\mathbb{N}$ indexed by $V$ and $W$ respectively, with finitely supported rows. Suppose $A$ and $B$ are representable shift equivalent with lag $m\in \mathbb{N}$ via $R$ and $S$, together with path isomorphisms $\phi$ and $\psi$. Then $X(A)_{\infty}$ and $X(B)_{\infty}$ are strong Morita shift equivalent with lag $m$ via $R,S$ and $\psi$. 
\end{Prop}

\begin{proof}
By Corollary \ref{c:cse-implies-rse} we have a Cuntz-Krieger family $(P_v,S_c)$ for $G_C$ generating $C^*(G_C)$ and a Cuntz-Krieger family $(P_v,T_d)$ for $G_D$ generating $C^*(G_D)$ on the same Hilbert space $\mathcal{H}$, satisfying equations \eqref{eq:psi} and \eqref{eq:phi}. Thus, we are in the context of the discussion above.

By the identifications preceding the theorem, we may define maps 
$$
U_{AR} : X(A)_{\infty} \otimes_{\A_{\infty}} X(R)_{\infty} \rightarrow X(A) \cdot X(R) \cdot \B_{\infty},
$$
$$
U_{RB} : X(R)_{\infty} \otimes_{\B_{\infty}} X(B)_{\infty} \rightarrow X(R) \cdot X(B) \cdot \B_{\infty},
$$ 
(where the notation $X(A) \cdot X(R) \cdot \B_{\infty}$ and $X(R) \cdot X(B) \cdot \B_{\infty}$ are understood as the closed linear span of products) by setting
$$
U_{AR}(S_a \otimes T_r \cdot w) = S_aT_r \cdot w, \ \ \text{and} \ \ U_{RB}(T_r\otimes S_b \cdot w) = T_rS_b \cdot w.
$$
for $a\in E_A, r \in E_r, b\in E_B$ and $w \in \B_{\infty}$. It is straightforward to show that $U_{AR}$ and $U_{RB}$ are well-defined unitary right $\B_{\infty}$-module maps. Thus, we are left with showing that $U_{AR}$ and $U_{RB}$ are left $\A_{\infty}$-module maps. 

We first show that $U_{AR}$ is a left $\A_{\infty}$-module map. We let $S_{\lambda}S_{\lambda'}^* \in \A_{\infty}$ for $\lambda,\lambda' \in E_A^{mk}$, and note that it will suffice to show that for $a\in E_A$ and $r\in E_R$ we have,
$$
U_{AR}(S_{\lambda}S_{\lambda'}^*(S_a \otimes T_r)) = S_{\lambda}S_{\lambda'}^*S_a T_r.
$$ 
To prove this, let $e\in E_A$ be some edge so that $s(e) = r(\lambda)=r(\lambda')$, and suppose that $\lambda = a_1 \nu$ and $\lambda' = a_1' \nu'$ for $a_1,a_1' \in E_A$. Write $\nu e = \psi(r_1s_1)\cdots \psi(r_ks_k)$ and $\nu'e = \psi(r_1's_1')\cdots \psi(r_k's_k')$ for $r_i,r_i' \in E_R$ and $s_i,s_i' \in E_S$. Then, we have 

\begin{flalign*}
& S_{\lambda}S_{\lambda'}^*S_a \otimes T_r = \delta_{a,a_1'} \cdot S_{a_1}S_{\nu}S_{\nu'}^* \otimes T_r = \delta_{a,a_1'} \cdot S_{a_1} \sum_{e\in s^{-1}(r(\lambda))}S_{\nu e}S_{\nu' e}^* \otimes T_r = &\\
& \delta_{a,a_1'} \cdot S_{a_1} \otimes \sum_{e\in s^{-1}(r(\lambda))} T_{r_1s_1\cdots r_ks_k}T_{r_1's_1'\cdots r_k's_k'}^* T_r = &\\
& \delta_{a,a_1'} \cdot S_{a_1} \otimes T_{r_1} \sum_{e\in s^{-1}(r(\lambda))} \delta_{r_1',r}\cdot T_{s_1\cdots r_ks_k}T_{s_1'\cdots r_k's_k'}^* = &\\
&\delta_{a,a_1'} \cdot S_{a_1} \otimes T_{r_1} \sum_{e\in s^{-1}(r(\lambda))} \delta_{r_1',r} \cdot \sum_{f \in s^{-1}(r(\lambda e))} S_{\psi_k^{-1}(s_1\cdots r_ks_k f)}S_{\psi_k^{-1}(s_1'\cdots r_k's_k' f)},
\end{flalign*} 
where in the above calculations we used the identifications via $\Psi_{\infty}^A$ and $\Psi_{\infty}^B$. Essentially the same calculation, using equation \eqref{eq:psi} instead, will show that
$$
S_{\lambda}S_{\lambda'}^*S_a T_r = 
$$
$$
\delta_{a,a_1'} \cdot S_{a_1} T_{r_1} \sum_{e\in s^{-1}(r(\lambda))} \delta_{r_1,r} \cdot \sum_{f \in s^{-1}(r(\lambda e))} S_{\psi_k^{-1}(s_1\cdots r_ks_k f)}S_{\psi_k^{-1}(s_1'\cdots r_k's_k' f)}.
$$
Thus, we see that $U_{AR}$ is a left $\A_{\infty}$-module map. 

Next, we show that $U_{RB}$ is a left $\A_{\infty}$-module map. We let $S_{\lambda}S_{\lambda'}^* \in \A_{\infty}$ with $\lambda,\lambda' \in E_A^{mk}$, and note that it will suffice to show that for $r\in E_R$ and $b\in E_B$ we have,
$$
U_{RB}(S_{\lambda}S_{\lambda'}^*T_r \otimes S_b) = S_{\lambda}S_{\lambda'}^*T_r S_b.
$$ 
Write $\lambda = \psi(r_1s_1) \cdots \psi(r_ks_k)$ and $\lambda' = \psi(r_1's_1')\cdots \psi(r_k's_k')$, and then further write $\psi(s_1r_2) = b_1 \sigma$ and $\psi(s_1'r_2') = b_1'\sigma'$ for $b_1,b_1' \in E_B$. Then, we have
$$
S_{\lambda}S_{\lambda'}^* T_r \otimes S_b = T_{r_1s_1 \cdots r_ks_k} T_{r_1's_1'\cdot r_k's_k'}^* T_r \otimes S_b =
$$
$$
\delta_{r_1',r} T_{r_1} \sum_{r \in r^{-1}(\lambda)} T_{s_1r_2 \cdots s_k r} T_{s_1'r_2' \cdots s_k'r}^* \otimes S_b =
$$
$$
\delta_{r_1',r} T_{r_1} \otimes \sum_{r \in r^{-1}(\lambda)} S_{\psi(s_1r_2) \cdots \psi (s_k r)}S_{\psi(s_1'r_2') \cdots \psi (s_k' r)}^*S_b =
$$
$$
\delta_{r_1',r} \delta_{b_1',b} \cdot T_{r_1} \otimes S_{b_1} \sum_{r \in r^{-1}(\lambda)}S_{\sigma \psi (s_2r_3) \cdots \psi(s_kr)} S_{\sigma' \psi(s_2'r_3') \cdots \psi(s_k'r)}^*,
$$
where in the above calculation we used the identifications via $\Psi_{\infty}^A$ and $\Psi_{\infty}^B$. Essentially the same calculation, using equation \eqref{eq:psi} instead, will show that
$$
S_{\lambda}S_{\lambda'}^* T_r S_b = 
$$
$$
\delta_{r_1',r} \delta_{b_1',b} \cdot T_{r_1} S_{b_1} \sum_{r \in r^{-1}(\lambda)}S_{\sigma \psi (s_2r_3) \cdots \psi(s_kr)} S_{\sigma' \psi(s_2'r_3') \cdots \psi(s_k'r)}^*.
$$
Thus, we see that $U_{BR}$ is a left $\A_{\infty}$-module map. 

Thus, to conclude the proof we need only show the equality
$$
X(A) \cdot X(R) \cdot \B_{\infty} = X(R) \cdot X(B) \cdot \B_{\infty}.
$$
However, it is clear from equation \eqref{eq:phi} that $X(A) \cdot X(R) = X(R) \cdot X(B)$, so we are done. Thus, the map $U_{RB}^{-1} \circ U_{AR}$ is a unitary isomorphism, showing that $X(R)_{\infty}$ is a strong Morita equivalence between $X(A)_{\infty}$ and $X(B)_{\infty}$.
\end{proof}

Suppose $A$ and $B$ are matrices indexed by $V$ and $W$ respectively, and suppose there are matrices $R$ and $S$, so that $C$ and $D$ are as described in the previous section. Suppose we have a path isomorphism $\psi : E_D^2 \rightarrow E_C^m$. 

Assume we have a unitary $U: X(A)_{\infty} \otimes X(R)_{\infty} \rightarrow X(R)_{\infty} \otimes X(B)_{\infty}$ such that for every $\xi \in X(A)_{\infty} \otimes X(R)_{\infty}$ and every $a \in \A_{\infty}$ and $b\in \B_{\infty}$ we have $U(b \xi a) = b U(\xi) a$. We refer the reader to \cite{AEE98} and \cite[Section 2]{MS00} for the basic theory of Morita equivalence, linking algebras and crossed products by Hilbert bimodules that we shall need in what follows. We denote by $\L$ the linking algebra of $X(R)_{\infty}$ given by
$$
\L:= \begin{bmatrix} 
		\A_{\infty} & X(R)_{\infty} \\
		X(R)_{\infty}^* & \B_{\infty} 
		\end{bmatrix},
$$
and by $W$ the $\L$-imprimitivity bimodule
$$
W:= \begin{bmatrix}
		X(A)_{\infty} & X(R)_{\infty} \otimes_{\B_{\infty}} X(B)_{\infty} \\
		X(R)_{\infty}^*\otimes_{\A_{\infty}} X(A)_{\infty} & X(B)_{\infty} 
		\end{bmatrix}.
$$
It is shown in the proof of \cite[Theorem 4.2]{AEE98} that there are two complementary and full projections $p_A, p_B \in \M(\O(W))$ (corresponding to the $2\times 2$ block form of $W$) together with two $*$-isomorphisms $\varphi_E : \O(X(A)_{\infty}) \rightarrow p_A \O(W)p_A$ and $\varphi_F : \O(X(B)_{\infty}) \rightarrow p_B \O(W) p_B$ given as follows. For $\xi \in X(A)_{\infty}$ and $\eta \in X(B)_{\infty}$ we let 
$$
\widetilde{\xi} := \begin{bmatrix} \xi & 0 \\
		0 & 0 
		\end{bmatrix} \ \ \text{and} \ \ \widetilde{\eta}:=\begin{bmatrix} 0 & 0 \\
		0 & \eta 
		\end{bmatrix},
		$$ 
so that the maps $\varphi_E$ and $\varphi_F$ are given by
$$
\varphi_E(S_{\xi}) = S_{\widetilde{\xi}} \ \ \text{and} \ \ \varphi_F(S_{\eta}) = S_{\widetilde{\eta}}.
$$
In particular, we see that $\varphi_E$ and $\varphi_F$ are gauge equivariant.

Since $W$ is an imprimitivity bimodule, again by \cite[Theorem 2.9]{AEE98} we get that $\O(W)_0 \cong \mathcal{L}$. In particular, we get that up to the same identification, the direct sum of diagonal subalgebras $\mathcal{D}_{E_A} \oplus \mathcal{D}_{E_B}$, according to $p_A, p_B \in \M(\O(W))$, coincides with
$$
\mathcal{D}_W:= \overline{\Span}\{ \ S_{\zeta}S_{\zeta'}^* \ | \ \zeta, \zeta' \in W^{\otimes n}, \ n \in \mathbb{N} \ \}.
$$
Hence, we get that $\varphi_E(\mathcal{D}_{E_A}) = p_A \mathcal{D}_W p_A$ and $\varphi_F(\mathcal{D}_{E_B}) = p_B \mathcal{D}_W p_B$. Let 
$$
N(\mathcal{D}_W) = \{ \ n \in \O(W) \ | \ n^*\mathcal{D}_W n , \ n \mathcal{D}_W n^* \subseteq \mathcal{D}_W \ \}
$$ 
be the set of normalizers of $\mathcal{D}_W$, and denote by
$$
N_{\star}(\mathcal{D}_W) = \{ \ n \in N(\mathcal{D}_W) \ | \ \exists k \in \mathbb{Z}, \ \gamma_z(n) = z^k n \ \}
$$ 
the \emph{homogeneous} normalizers of $\mathcal{D}_W$.

\begin{Prop} \label{p:x(r)-implies-esdi}
Let $A$ and $B$ be essential matrices over $\mathbb{N}$, indexed by $V$ and $W$ respectively, with finitely supported rows. Suppose that $X(A)_{\infty}$ and $X(B)_{\infty}$ are strong Morita shift equivalent. Then there is an equivariant $*$-isomorphism $\varphi : C^*(G_A) \otimes \bbK \rightarrow C^*(G_B) \otimes \bbK$ with $\varphi(\mathcal{D}_{E_A} \otimes c_0) = \mathcal{D}_{E_B} \otimes c_0$.
\end{Prop}

\begin{proof}
We are in the situation where we can apply the implication $(7) \implies (8)$ in \cite[Corollary 11.3]{CRST17}. By the description of $C^*(G_A) \cong C^*(\G_{E_A})$ and $C^*(G_B) \cong C^*(\G_{E_B})$ as groupoid C*-algebras, we get that item $(8)$ in \cite[Corollary 11.3]{CRST17} is equivalent to $C^*(G_A)$ and $C^*(G_B)$ being stably equivariant diagonal-isomorphic. Indeed, this is because $\mathbb{Z}$-coactions correspond to topological $\mathbb{Z}$-gradings by \cite[Remark 6]{Rae18}, which in turn correspond to $\mathbb{T}$-actions by \cite[Theorem 3]{Rae18}. This means that the second equality in item $(8)$ of \cite[Corollary 11.3]{CRST17} is equivalent to equivariance of the isomorphism.

Thus, to prove our result it will suffice to prove item $(7)$ in \cite[Corollary 11.3]{CRST17}. Everything is set up in the discussion preceding the proposition, except for one thing. We are left with showing that $p_A \O(W) p_B$ is the closed linear span of $p_A N_*(\mathcal{D}_W) p_B$. However, we know that $\O(W)$ is the closed linear span of its graded subspaces $\O(W)_n$ for $n\in \mathbb{Z}$. Hence, $p_A \O(W) p_B$ is the closed linear span of $p_A\O(W)_np_B$. By \cite[Theorem 2.9]{AEE98} we have for all $n \in \mathbb{Z}$ that
\begin{enumerate}
\item
$\O(W)_n \cong W^{\otimes n}$ for $n > 0$,

\item
$\O(W)_0 \cong \mathcal{L}$, and

\item
$\O(W)_n \cong (W^{\otimes n})^*$ for $n < 0$.
\end{enumerate}

Thus, we are left with showing that each $p_A \O(W)_n p_B$ is the closed linear span of its normalizers (which are automatically $n$-homogeneous). 

For $n=0$ we get that $p_A \O(W)_0 p_B \cong X(R)_{\infty}$, for $n > 0$ we get that $p_A \O(W)_n p_B \cong X(A)_{\infty}^n \otimes X(R)_{\infty}$, and for $n< 0$ we get that $p_A \O(W)_n p_B \cong (X(A)_{\infty}^*)^n \otimes X(R)_{\infty}$.

Thus, $p_A \O(W)_n p_B$ is the closed linear span of elements of the form $S_{\lambda}S_{\lambda'}^* S_{\alpha} T_r$ (for $n \geq 0$) or $S_{\lambda}S_{\lambda'}^* S_{\alpha}^* T_r$ (for $n<0$) for $\lambda,\lambda' \in E_A^{\ell}$, $\alpha \in E_A^{|n|}$ and $r\in E_R$ for some $\ell \in \mathbb{N}$. In order to show that these elements are normalizers for $\mathcal{D}_W$, and since $S_{\lambda}S_{\lambda'}^* S_{\alpha}, S_{\lambda}S_{\lambda'}^* S_{\alpha}^* \in N(\mathcal{D}_{E_A})$, it will suffice to show that $T_r\mathcal{D}_{E_B} T_r^* \subseteq \mathcal{D}_{E_A}$ and that $T_r^* \mathcal{D}_{E_A}T_r \subseteq \mathcal{D}_{E_B}$ for $r\in E_R$.

We know already from the inclusions $\mathcal{D}_{E_A} \subseteq \A_{\infty}$ and $\mathcal{D}_{E_B} \subseteq \B_{\infty}$ that $\mathcal{D}_{E_A}$ is the closed linear span of elements of the form 
$$
S_{\psi_k(r_1s_1 \cdots r_ks_k)}S_{\psi_k(r_1s_1\cdots r_ks_k)}^*
$$ 
for paths $r_1s_1\cdots r_ks_k \in E_D^{2k}$ with $r_i \in E_R$ and $s_i \in E_S$, and that $\mathcal{D}_{E_B}$ is the closed linear span of elements of the form 
$$
S_{\psi_k(s_1r_1 \cdots s_kr_k)}S_{\psi_k(s_1r_1\cdots s_kr_k)}^*
$$ for paths $s_1r_1\cdots s_kr_k \in E_D^{2k}$ with $r_i \in E_R$ and $s_i \in E_S$.
So for such paths we compute,
$$
T_rS_{\psi_k(s_1r_1 \cdots s_kr_k)}S_{\psi_k(s_1r_1 \cdots s_kr_k)}^* T_r^* = \sum_{s\in E_S} T_rT_{s_1r_1\cdots s_kr_k s}T_{s_1r_1\cdots s_kr_ks}^* T_r^*
$$
$$
= \sum_{s\in E_S} S_{\psi_{k+1}(rs_1r_1\cdots s_kr_ks)}S_{\psi_{k+1}(rs_1r_1\cdots s_kr_ks)}^*
$$
is in $\mathcal{D}_{E_B}$. On the other hand, 
$$
T_r^*S_{\psi_k(r_1s_1 \cdots r_ks_k)}S_{\psi_k(r_1s_1\cdots r_ks_k)}^* T_r = \delta_{r,r_1} \cdot \sum_{r \in E_R} S_{\psi_k(s_1 \cdots r_ks_kr)}S_{\psi_k(s_1\cdots r_ks_kr)}^*
$$
is in $\mathcal{D}_{E_A}$. Thus, $T_r\mathcal{D}_{E_B} T_r^* \subseteq \mathcal{D}_{E_A}$ and $T_r^* \mathcal{D}_{E_A}T_r \subseteq \mathcal{D}_{E_B}$ as required.
\end{proof}

\section{Shift equivalences through the lens} \label{s:final}

In this section we orient various equivalence relations between strong shift equivalence and shift equivalence. We will assume some familiarity with crossed product C*-algebras and K-theory of C*-algebras. We recommend \cite{Wil07} for the basic theory of crossed product C*-algebras, \cite{LLR00} for the basic K-theory for C*-algebras, and especially \cite[Chapter 7]{Rae05} for K-theory and crossed products of graph algebras by their gauge actions.

Suppose now that $G=(V,E)$ is a directed graph, and let $\gamma$ be the gauge unit circle action on $C^*(G)$. We denote by $\widehat{\gamma}$ the dual $\mathbb{Z}$ action on $C^*(G) \rtimes_{\gamma} \mathbb{T}$. By \cite[Lemma 2.75]{Wil07} we have that $[C^*(G)\rtimes \bbT] \otimes \mathbb{K}$ and $[C^*(G)\otimes \mathbb{K}] \rtimes \bbT$ are $*$-isomorphic via the map $f \otimes K \mapsto f \cdot K$ defined for $f\in C(\bbT; C^*(G))$ and $K\in \mathbb{K}$, and that this map intertwines the action $\widehat{\gamma}^G \otimes \id$ with the dual action $(\gamma^G\otimes \id)^{\widehat{\mbox{}}}$. Using this equivariant identification together with \cite[Corollary 2.48]{Wil07} we obtain the following standard fact which we leave for the reader to verify.

\begin{Prop} \label{p:stequivariant-implies-stequivariant}
Let $G$ and $G'$ be directed graphs, and suppose there exists a $*$-isomorphism $\varphi: C^*(G) \otimes \mathbb{K} \rightarrow C^*(G') \otimes \mathbb{K}$ such that $\varphi \circ (\gamma^G \otimes \id)_z = (\gamma^{G'}\otimes \id)_z \circ \varphi$ for all $z\in \bbT$. Then there exists a $*$-isomorphism $\varphi \rtimes \id : [C^*(G)\rtimes_{\gamma^G} \bbT] \otimes \mathbb{K} \rightarrow [C^*(G')\rtimes_{\gamma^G} \bbT] \otimes \mathbb{K}$ such that $[\varphi \rtimes \id] \circ [\widehat{\gamma}^G \otimes \id] = [\widehat{\gamma}^{G'} \otimes \id] \circ [\varphi \rtimes \id]$.
\end{Prop}

Let $A$ be a finite essential matrix with entries in $\mathbb{N}$. We denote by $(D_A,D_A^+)$ the inductive limit of the following inductive system of ordered abelian groups acting on columns
\[\begin{CD}
(\bbZ^V,\bbZ^V_+) @>A^T>> (\bbZ^V,\bbZ^V_+) @>A^T>> (\bbZ^V,\bbZ^V_+) @>A^T>> \cdots
\end{CD}\]
and let $d_A: D_A \to D_A$ be the homomorphism induced by the diagram
\[\begin{CD}
	\bbZ^V @>A^T>> \bbZ^V @>A^T>> \bbZ^V @>A^T>> \cdots D_A \\
	@VVA^T V @VVA^T V @VVA^T V @VV d_AV  \\
	\bbZ^V @>A^T>> \bbZ^V @>A^T>> \bbZ^V @>A^T>> \cdots D_A
\end{CD}\]
The triple $(D_A,D_A^+, d_A)$ is called the \emph{dimension group triple} of $A$. 

Let $H_i$ be an abelian group, $H_i^+$ a submonoid of $H_i$ and $\alpha_i$ is an automorphism of $H_i$ for $i=1,2$. We say two triples $(H_i,H_i^+,\alpha_i)$ are isomorphic if there is a group isomorphism $\phi:H_1\to H_2$ such that $\phi(H_1^+)=H_2^+$ and $\phi\circ\alpha_1=\alpha_2\circ\phi$. 

From work of Wagoner \cite[Corollary 2.9]{Wag88} we have that $(D_A,D_A^+, d_A)$ is isomorphic to a dimension group triple constructed from the edge shift $(X_{E_A},\sigma_{E_A})$ associated to $A$. On the other hand, the dimension group triple $(D_A,D_A^+, d_A)$ also coincides with a K-theory triple arising from the crossed product C*-algebra $C^*(G_A)\rtimes_\ga \mathbb{T}$ where $\ga$ is the gauge unit circle action. More precisely, noting that arrow directions in \cite{Rae05} are reversed to ours, it follows from \cite[Corollary 7.14]{Rae05} together with the discussion preceding \cite[Lemma 7.15]{Rae05} that the triple $(D_A,D_A^+, d_A)$ is isomorphic to the triple 
$$
(K_0(C^*(G_A)\rtimes_\ga\bbT),K_0(C^*(G_A)\rtimes_\ga\bbT)^+,K_0(\hga_1)^{-1}),
$$
where $\hga$ is the dual action of $\mathbb{Z}$ on $C^*(G_A)\rtimes_\ga\bbT$. The implication $(3) \implies (4)$ below is what we referred to as \emph{Krieger's corollary} in the introduction.

\begin{Cor} \label{c:intermediates}
Suppose now that $A$ and $B$ are two finite essential matrices with entries in $\mathbb{N}$. The former conditions imply the latter
\begin{enumerate}

\item $A$ and $B$ are strong shift equivalent.

\item $C^*(G_A)$ and $C^*(G_B)$ are equivariantly stably isomorphic in a way that respects the diagonals.

\item $C^*(G_A)$ and $C^*(G_B)$ are equivariantly stably isomorphic.

\item $A$ and $B$ are shift equivalent.
\end{enumerate}
\end{Cor}

\begin{proof}
By Williams' theorem \cite[Theorem 7.5.8]{Wil73} (cf. \cite{LM95}), strong shift equivalence of $A$ and $B$ coincides with conjugacy of two-sided edge shifts $(X_{E_A},\sigma_{E_A})$ and $(X_{E_B},\sigma_{E_B})$ (see \cite[Theorem 7.2.7]{LM95}), so we get that $(1)$ implies $(2)$ by the implication $(III) \implies(II)$ of \cite[Theorem 5.1]{CR17} (see also \cite[Proposition 2.17]{CK80}). 

Clearly $(2) \implies (3)$, so we are left with showing $(3) \implies (4)$. To do this, we use Proposition 
\ref{p:stequivariant-implies-stequivariant} to get that $[C^*(G)\rtimes_{\gamma^A} \bbT]\otimes\bbK$ and $[C^*(G') \rtimes_{\gamma^B}\bbT] \otimes\bbK$ are equivariantly isomorphic with actions $\widehat{\gamma}^A\otimes \id$ and $\widehat{\gamma}^B \otimes \id$ respectively. Hence, by applying K-theory we get an isomorphism of the triples 
\begin{gather*}
(K_0(C^*(G) \rtimes_{\gamma^A}\bbT),K_0(C^*(G)\rtimes_{\gamma^A}\bbT)^+,K_0(\widehat{\gamma}^A_1)) \ \text{ and } \\ (K_0(C^*(G')\rtimes_{\gamma^B} \bbT),K_0(C^*(G')\rtimes_{\gamma^B} \bbT)^+,K_0(\widehat{\gamma}^B_1)).
\end{gather*}
By Krieger's theorem \cite[Theorem 6.4]{Eff81} (see also \cite{Kri80}), and up to the identification with dimension triples, we get that the triples above are isomorphic if and only if $A$ and $B$ are shift equivalent. Hence, $(3) \implies (4)$.
\end{proof} 

From the perspective of C*-algebras, we get that compatible, representable, and strong Morita shift equivalences are between strong shift equivalence and shift equivalence. The work done in the previous sections allows us to orient them, and show that each one provides a new way of obstructing SSE when merely assuming SE. The equivalence between $(1)$ and $(4)$ in the theorem below realizes a goal sought after by Muhly, Pask and Tomforde in \cite[Remark 5.5]{MPT08}, and provides a characterization of strong shift equivalence of matrices in terms of strong Morita shift equivalence.

\begin{Thm} \label{t:main}
Suppose $A$ and $B$ are two finite essential matrices with entries in $\mathbb{N}$. Then the following are equivalent
\begin{enumerate}
\item $A$ and $B$ are strong shift equivalent.
\item $A$ and $B$ are compatibly shift equivalent.
\item $A$ and $B$ are representable shift equivalent.
\item $X(A)_{\infty}$ and $X(B)_{\infty}$ are strong Morita shift equivalent.
\item $C^*(G_A)$ and $C^*(G_B)$ are equivariantly stably isomorphic in a way that respects the diagonals.
\end{enumerate}
\end{Thm}

\begin{proof}
It follows from Corollary \ref{c:sse_implies_cse} that $(1) \implies (2)$, from Corollary \ref{c:cse-implies-rse} that $(2) \implies (3)$, and from Proposition \ref{p:rse-implies-sme} that $(3) \implies (4)$. The implication $(4) \implies (5)$ is provided by Proposition \ref{p:x(r)-implies-esdi}. Finally, the implication $(5) \implies (1)$ is granted to us by combining the implication $(II) \implies (III)$ of \cite[Theorem 5.1]{CR17} and the fact that conjugacy of two-sided subshifts coincides with strong shift equivalence (Williams' Theorem \cite{Wil73}).
\end{proof}

We note here that if the construction in \cite[Section 5]{KK14} is applied to $E:= X(A)$, $F:= X(B)$ and $X := X(D)$ (as in the notation of \cite{KK14}), we obtain again the correspondences $E_{\infty}$, $F_{\infty}$ and $X_{\infty}$ as in \cite[Section 5]{KK14}. This follows from uniqueness of Pimsner dilations \cite[Theorem 3.9]{KK14}. In particular, the correspondence $X(R)_{\infty}$ coincides with ``$R_{\infty}$" in \cite[Section 5]{KK14} as the $1-2$ corner of $X_{\infty} = X(D)_{\infty}$. 

Hence, when we specify to $E:= X(A)$ and $F:= X(B)$, in the proof of \cite[Theorem 5.8]{KK14} it is erroneously claimed that $X(R)_{\infty}$ is the imprimitivity bimodule that implements a strong Morita equivalence between $X(A)_{\infty}$ and $X(B)_{\infty}$, through the identifications coming from $\psi_A$ and $\psi_B$. However, by our result this would show that shift equivalence implies strong shift equivalence. Hence, the strategy of proof in \cite[Theorem 5.8]{KK14} cannot be made to work, as the following cutoff result demonstrates.

\begin{Thm} \label{t:cutoff}
There exist finite aperiodic irreducible matrices $A$ and $B$ with entries in $\mathbb{N}$ that such that $X(A)_{\infty}$ and $X(B)_{\infty}$ are strong Morita equivalent, but not strong Morita shift equivalent.
\end{Thm}

\begin{proof}
Let $A$ and $B$ be the aperiodic and irreducible counterexamples of Kim and Roush from \cite{KR99}, so that they are shift equivalent but not strong shift equivalent. By the result of Bratteli and Kishimoto \cite{BK00} we know that $C^*(G_A)$ and $C^*(G_B)$ are equivariantly stably isomorphic, so that by Theorem \ref{t:down-and-up} we get that $X(A)_{\infty}$ and $X(B)_{\infty}$ are strong Morita equivalent.

On the other hand, since $A$ and $B$ are not strong shift equivalent, by Theorem \ref{t:main} we get that $X(A)_{\infty}$ and $X(B)_{\infty}$ cannot be strong Morita shift equivalent.
\end{proof}

\begin{Rmk}
Suppose $A$ and $B$ are aperiodic and irreducible matrices over $\mathbb{N}$, and not $1 \times 1$. Then the following conditions are equivalent
\begin{enumerate}
\item $A$ and $B$ are shift equivalent.
\item $X(A)$ and $X(B)$ are shift equivalent.
\item $X(A)_{\infty}$ and $X(B)_{\infty}$ are strong Morita equivalent.
\item $C^*(G_A)$ and $C^*(G_B)$ are equivariantly stably isomorphic.
\item The triples $(D_A,D_A^+,d_A)$ and $(D_B,D_B^+, d_B)$ are isomorphic.
\end{enumerate}
Indeed, Krieger's theorem shows that $(1) \iff (5)$, Theorem \ref{p:shift-equiv-generalizes} shows that $(1) \iff (2)$, Theorem \ref{t:down-and-up} shows that $(3)\iff (4)$, and Corollary \ref{c:intermediates} shows that $(1) \implies (4)$. Finally, from Bratteli and Kishimoto \cite[Corollary 4.3]{BK00} we get $(5) \implies (4)$, which finishes the proof.

Note that the above proof avoids the implication $(2) \implies (3)$. Since now the validity of \cite[Theorem 5.8]{KK14} is in question, it is unknown whether one can prove $(2) \implies (3)$ directly, without classification techniques as in \cite{BK00}.
\end{Rmk}

\subsection*{Acknowledgments} The authors are grateful to Mike Boyle for bringing their attention to the work of Parry \cite{Par81}, as well as to Kevin Brix and Efren Ruiz for some remarks on earlier versions of this paper. The authors are also thankful for suggestions and remarks made by anonymous referees, ultimately leading to a more streamlined and readable version of the paper.


\end{document}